\documentclass[12pt]{article}

\usepackage{amsthm}
\usepackage{amscd}
\usepackage{amsfonts}
\usepackage{amssymb}
\usepackage{amsgen}
\usepackage{amsmath}
\usepackage{amsopn}
\usepackage{verbatim}
\usepackage{xypic}
\usepackage{xspace}
\usepackage{multicol}
\usepackage{url}
\usepackage{upref}

\theoremstyle{plain}
\newtheorem{thm}{Theorem}[section]

\newtheorem{cor}[thm]{Corollary}

\theoremstyle{definition}

\theoremstyle{remark}
\newtheorem{rem}[thm]{Remark}

 \font\cyr=wncyr10
 \newcommand{\nc}{\newcommand}

\nc{\per}[1]{\underset{#1}{\boldsymbol \pi}\,}

 \nc{\MT}{{\rm MT}}
 \nc{\XX}{{X}}
 \nc{\gF}{{\varPhi}}
 \nc{\ot}{\otimes}

 \nc{\wht}{\widehat}
 \nc{\bwg}{{\bigwedge}}
 \nc{\wg}{{\wedge}}
 \nc{\mmu}{{\boldsymbol{\mu}}}
 \nc{\mal}{{{\scriptstyle \maltese}}}
 \nc{\fA}{{\mathfrak A}}
 \nc{\HH}{{\mathfrak H}}
 \nc{\ra}{\rightarrow}
 \nc{\ors}{{\bfs}}
 \nc{\orr}{{\bfr}}
 \nc{\os}{{\overset}}
 \nc{\G}{{\mathbb G}}
 \nc{\F}{{\mathbb F}}
 \nc{\Z}{{\mathbb Z}}
 \nc{\R}{{\mathbb R}}
 \nc{\N}{{\mathbb N}}
 \nc{\ZN}{{\mathbb Z_{\ge 0}}}
 \nc{\Q}{{\mathbb Q}}
 \nc{\C}{{\mathbb C}}
 \nc{\CP}{{\mathbb{CP}}}
 \nc{\Cnn}{{\mathbb C}_{\ge 0}}
 \nc{\Cp}{{\mathbb C}_{>0}}
 \nc{\MPV}{{\mathcal{MPV}}}
 
 \nc{\tB}{{\tilde B}}

 \nc{\suf}{{\ast\,}}
 \nc{\sufq}{{\ast_q\,}}
 \nc{\gam}{{\gamma}}
 \nc{\gG}{{\Gamma}}
 \nc{\om}{{\omega}}
 \nc{\vep}{{\varepsilon}}
 \nc{\ga}{{\alpha}}
 \nc{\gl}{{\lambda}}
 \nc{\gb}{{\beta}}
 \nc{\gf}{{\varphi}}
 \nc{\gd}{{\delta}}
 \nc{\orgd}{{\vec \gd\,}}
 \nc{\gs}{{\sigma}}
 \nc{\gth}{{\theta}}
 \nc{\gS}{{\Sigma}}

 \nc{\gk}{{\kappa}}
  \nc{\gz}{{\zeta}}
 \nc{\tgz}{{\tilde{\zeta}}}
 \nc{\gO}{{\Omega}}
 \nc{\sif}{{\mathcal S}}
 \nc{\gt}{{\tau}}
 \nc{\Lra}{\Longrightarrow}
 \nc{\lra}{\longrightarrow}
 \nc{\lmaps}{\longmapsto}
 \nc{\fS}{{\mathfrak S}}
 \nc{\DD}{{\mathfrak D}}
 \nc{\Llra}{\Longleftrightarrow}
 \nc{\ol}{\overline}
 \nc{\ola}{\overleftarrow}
 \nc{\lms}{\longmapsto}
 \nc{\cv}{{{\mathsf c}{\mathsf v}}}
 \nc{\zq}{{\zeta_q}}
 \nc\qup{{q\uparrow 1}}
 \nc{\us}{\underset}
 \nc{\tn}{{\tilde{n}}}
 \nc{\gD}{{\Delta}}
 \nc{\bi}{{\bf i}}
 \nc{\bfone}{{\bf 1}}

 \nc{\bfa}{{\bf a}}
 \nc{\bfb}{{\bf b}}
 \nc{\bfc}{{\bf c}}
 \nc{\bfd}{{\bf d}}
 \nc{\bfe}{{\bf e}}
 \nc{\bff}{{\bf f}}
 \nc{\bfg}{{\bf g}}
 \nc{\bfi}{{\bf i}}
 \nc{\bfj}{{\bf j}}

 \nc{\bfn}{{\bf n}}
 \nc{\bfl}{{\bf l}}
 \nc{\bfk}{{\bf k}}
 \nc{\bfm}{{\bf m}}
 \nc{\bfo}{{\bf o}}
 \nc{\bfp}{{\bf p}}
 \nc{\bfq}{{\bf q}}
 \nc{\bfr}{{\bf r}}
 \nc{\bfs}{{\bf s}}
 \nc{\bft}{{\bf t}}
 \nc{\bfu}{{\bf u}}
 \nc{\bfv}{{\bf v}}
 \nc{\bfw}{{\bf w}}
 \nc{\bfx}{{\bf x}}
 \nc{\bfy}{{\bf y}}
 \nc{\bfz}{{\bf z}}
 \nc{\bfB}{{\bf B}}
 \nc{\bfP}{{\bf P}}
 \nc{\bfQ}{{\bf Q}}
 \nc{\bfY}{{\bf Y}}
 \nc{\bfgb}{{\boldsymbol \gb}}
 \nc{\bfga}{{\boldsymbol \ga}}
 \nc{\bfrho}{{\boldsymbol \rho}}
 \nc{\bfchi}{{\boldsymbol \chi}}
 \nc{\QX}{{\Q\langle \bfX\rangle}}
 \nc{\QY}{{\Q\langle \bfY\rangle}}
 \nc{\CX}{{\C\langle \bfX\rangle}}
 \nc{\CY}{{\C\langle \bfY\rangle}}
 \nc{\QXX}{{\Q\langle\!\langle \bfX\rangle\!\rangle}}
 \nc{\QYY}{{\Q\langle\!\langle \bfY\rangle\!\rangle}}
 \nc{\CXX}{{\C\langle\!\langle \bfX\rangle\!\rangle}}
 \nc{\CYY}{{\C\langle\!\langle \bfY\rangle\!\rangle}}

 \nc{\bbA}{{\mathbb A}}
 \nc{\bbB}{{\mathbb B}}
 \nc{\bbC}{{\mathbb C}}
 \nc{\bbD}{{\mathbb D}}
 \nc{\bbE}{{\mathbb E}}
 \nc{\bbF}{{\mathbb F}}
 \nc{\bbG}{{\mathbb G}}
 \nc{\bbH}{{\mathbb H}}
 \nc{\bbI}{{\mathbb I}}
 \nc{\bbJ}{{\mathbb J}}
 \nc{\bbK}{{\mathbb K}}
 \nc{\bbL}{{\mathbb L}}
 \nc{\bbM}{{\mathbb M}}
 \nc{\bbN}{{\mathbb N}}
 \nc{\bbO}{{\mathbb O}}
 \nc{\bbP}{{\mathbb P}}
 \nc{\bbQ}{{\mathbb Q}}
 \nc{\bbR}{{\mathbb R}}
 \nc{\bbS}{{\mathbb S}}
 \nc{\bbT}{{\mathbb T}}
 \nc{\bbU}{{\mathbb U}}
 \nc{\bbV}{{\mathbb V}}
 \nc{\bbW}{{\mathbb W}}
 \nc{\bbX}{{\mathbb X}}
 \nc{\bbY}{{\mathbb Y}}
 \nc{\bbZ}{{\mathbb Z}}
 \nc{\bba}{{\mathbb a}}
 \nc{\bbb}{{\mathbb b}}
 \nc{\bbc}{{\mathbb c}}
 \nc{\bbd}{{\mathbb d}}
 \nc{\bbe}{{\mathbb e}}
 \nc{\bbf}{{\mathbb f}}
 \nc{\bbg}{{\mathbb g}}
 \nc{\bbh}{{\mathbb h}}
 \nc{\bbi}{{\mathbb i}}
 \nc{\bbk}{{\mathbb k}}
 \nc{\bbl}{{\mathbb l}}
 \nc{\bbm}{{\mathbb m}}
 \nc{\bbn}{{\mathbb n}}
 \nc{\bbo}{{\mathbb o}}
 \nc{\bbp}{{\mathbb p}}
 \nc{\bbq}{{\mathbb q}}
 \nc{\bbr}{{\mathbb r}}
 \nc{\bbs}{{\mathbb s}}
 \nc{\bbt}{{\mathbb t}}
 \nc{\bbu}{{\mathbb u}}
 \nc{\bbv}{{\mathbb v}}
 \nc{\bbw}{{\mathbb w}}
 \nc{\bbx}{{\mathbb x}}
 \nc{\bby}{{\mathbb y}}
 \nc{\bbz}{{\mathbb z}}

 \nc{\MZV}{{\mathcal{MZV}}}
 \nc{\calA}{{\mathcal A}}
 \nc{\calB}{{\mathcal B}}
 \nc{\calC}{{\mathcal C}}
 \nc{\calD}{{\mathcal D}}
 \nc{\calE}{{\mathcal E}}
 \nc{\calF}{{\mathcal F}}
 \nc{\calG}{{\mathcal G}}
 \nc{\calH}{{\mathcal H}}
 \nc{\calI}{{\mathcal I}}
 \nc{\calJ}{{\mathcal J}}
 \nc{\calK}{{\mathcal K}}
 \nc{\calL}{{\mathcal L}}
 \nc{\calM}{{\mathcal M}}
 \nc{\calN}{{\mathcal N}}
 \nc{\calO}{{\mathcal O}}
 \nc{\calP}{{\mathcal P}}
 \nc{\calQ}{{\mathcal Q}}
 \nc{\calR}{{\mathcal R}}
 \nc{\calS}{{\mathcal S}}
 \nc{\calT}{{\mathcal T}}
 \nc{\calU}{{\mathcal U}}
 \nc{\calV}{{\mathcal V}}
 \nc{\calW}{{\mathcal W}}
 \nc{\calX}{{\mathcal X}}
 \nc{\calY}{{\mathcal Y}}
 \nc{\calZ}{{\mathcal Z}}
  \nc{\cala}{{\mathcal a}}
 \nc{\calb}{{\mathcal b}}
 \nc{\calc}{{\mathcal c}}
 \nc{\cald}{{\mathcal d}}
 \nc{\cale}{{\mathcal e}}
 \nc{\calf}{{\mathcal f}}
 \nc{\calg}{{\mathcal g}}
 \nc{\calh}{{\mathcal h}}
 \nc{\cali}{{\mathcal i}}
 \nc{\calj}{{\mathcal j}}
 \nc{\calk}{{\mathcal k}}
 \nc{\call}{{\mathcal l}}
 \nc{\calm}{{\mathcal m}}
 \nc{\caln}{{\mathcal n}}
 \nc{\calo}{{\mathcal o}}
 \nc{\calp}{{\mathsf p}}
 \nc{\calq}{{\mathcal q}}
 \nc{\calr}{{\mathcal r}}
 \nc{\cals}{{\mathcal s}}
 \nc{\calt}{{\mathcal t}}
 \nc{\calu}{{\mathcal u}}
 \nc{\calv}{{\mathcal v}}
 \nc{\calw}{{\mathcal w}}
 \nc{\calx}{{\mathcal x}}
 \nc{\caly}{{\mathcal y}}
 \nc{\calz}{{\mathcal z}}

 \nc{\frakA}{{\mathfrak A}}
 \nc{\frakB}{{\mathfrak B}}
 \nc{\frakC}{{\mathfrak C}}
 \nc{\frakD}{{\mathfrak D}}
 \nc{\frakE}{{\mathfrak E}}
 \nc{\frakF}{{\mathfrak F}}
 \nc{\frakG}{{\mathfrak G}}
 \nc{\frakH}{{\mathfrak H}}
 \nc{\frakI}{{\mathfrak I}}
 \nc{\frakJ}{{\mathfrak J}}
 \nc{\frakK}{{\mathfrak K}}
 \nc{\frakL}{{\mathfrak L}}
 \nc{\frakM}{{\mathfrak M}}
 \nc{\frakN}{{\mathfrak N}}
 \nc{\frakO}{{\mathfrak O}}
 \nc{\frakP}{{\mathfrak P}}
 \nc{\frakQ}{{\mathfrak Q}}
 \nc{\frakR}{{\mathfrak R}}
 \nc{\frakS}{{\mathfrak S}}
 \nc{\frakT}{{\mathfrak T}}
 \nc{\frakU}{{\mathfrak U}}
 \nc{\frakV}{{\mathfrak V}}
 \nc{\frakW}{{\mathfrak W}}
 \nc{\frakX}{{\mathfrak X}}
 \nc{\frakY}{{\mathfrak Y}}
 \nc{\frakZ}{{\mathfrak Z}}
 \nc{\fraka}{{\mathfrak a}}
 \nc{\frakb}{{\mathfrak b}}
 \nc{\frakc}{{\mathfrak c}}
 \nc{\frakd}{{\mathfrak d}}
 \nc{\frake}{{\mathfrak e}}
 \nc{\frakf}{{\mathfrak f}}
 \nc{\frakg}{{\mathfrak g}}
 \nc{\frakh}{{\mathfrak h}}
 \nc{\fraki}{{\mathfrak i}}
 \nc{\frakj}{{\mathfrak j}}
 \nc{\frakk}{{\mathfrak k}}
 \nc{\frakl}{{\mathfrak l}}
 \nc{\frakm}{{\mathfrak m}}
 \nc{\frakn}{{\mathfrak n}}
 \nc{\frako}{{\mathfrak o}}
 \nc{\frakp}{{\mathfrak p}}
 \nc{\frakq}{{\mathfrak q}}
 \nc{\frakr}{{\mathfrak r}}
 \nc{\fraks}{{\mathfrak s}}
 \nc{\frakt}{{\mathfrak t}}
 \nc{\fraku}{{\mathfrak u}}
 \nc{\frakv}{{\mathfrak v}}
 \nc{\frakw}{{\mathfrak w}}
 \nc{\frakx}{{\mathfrak x}}
 \nc{\fraky}{{\mathfrak y}}
 \nc{\frakz}{{\mathfrak z}}
 \nc{\so}{{\mathfrak so}}
 \nc{\sa}{{\mbox{{\scriptsize \cyr x}}}}
 \nc{\slfour}{{\mathfrak sl}_4}
 \nc{\one}{{\bf 1}}
 \nc{\zero}{{\bf 0}}
 \nc{\Qxy}{\Q\langle x,y\rangle}

\begin{document}

\title{New Families of Weighted Sum Formulas for \\ Multiple Zeta Values}

\author{Haiping Yuan \& Jianqiang Zhao}

\date{}
\maketitle

\noindent \textbf{Abstract.}
In this paper we use the generating functions and the double shuffle relations satisfied by the multiple zeta values to derive some new families of identities.

\section{Introduction}
In recent years there is a flux of research on the multiple zeta functions and their special values due to their deep
connections with many branches of mathematics and physics. For any positive integer $d$ (called the \emph{depth}) and $s_1,\dots,s_d$
with $s_1>1$ the multiple zeta values (MZVs) are defined by
\begin{equation*}
\zeta(s_1,\dots,s_d)=\sum_{ k_1>\dots>k_d>0} \frac{1}{k_1^{s_1}\cdots k_d^{s_d}}.
\end{equation*}
These values are easily seen to satisfy the so called stuffle relation. For example,
\begin{equation}\label{equ:2RiemStuffle2}
 \zeta(s_1)\zeta(s_2)=\zeta(s_1,s_2)+\zeta(s_2,s_1)+\zeta(s_1+s_2).
\end{equation}
Euler \cite{Euler1775} first studied the depth two case and obtained the following
decomposition formula  by using partial fraction techniques.
\begin{equation}\label{equ:2RiemZetaProd}
\zeta(s_1) \zeta(s_2)
 = \sum_{\substack{t_1\ge 2,t_2\ge 1\\ t_1+t_2=s_1+s_2}}
\bigg[\binom{t_1-1}{s_1-1}+\binom{t_2-1}{s_2-1}\bigg] \zeta(t_1,t_2), \quad s_1,s_2\ge 2.
\end{equation}
Although he did not consider divergence problem his approach has been made rigorous using modern techniques of regularization. Similarly to \eqref{equ:2RiemZetaProd} one can show that for all $s_1,s_3\ge 2$ and $s_2\ge 1$,
\begin{multline}
 \label{equ:2Riem1RiemProd}
\zeta(s_1,s_2)\,\zeta(s_3)  =
\sum_{\substack{t_1\ge 2,t_2\ge 1\\ t_1+t_2=s_1+s_3}} \binom{t_1-1}{s_3-1}\zeta(t_1,t_2,s_2)\\
+ \sum_{\substack{t_1\ge 2,t_2,t_3\ge 1\\ t_1+t_2+t_3 \\  =s_1+s_2+s_3 }}
\binom{t_1-1}{s_1-1} \bigg[\binom{t_2-1}{s_2-t_3}  +\binom{t_2-1}{s_2-1} \bigg]\zeta(t_1,t_2,t_3).
\end{multline}

In fact, nowadays this can be derived easily by the shuffle relations satisfied by the
iterated integral expression of MZVs (see \cite[p.\ 510]{Zagier1994}).
By combining the stuffle and the shuffle
relations one can obtain the so called double shuffle relations (see \cite{IKZ2006} for details).

Let $d$ be any positive integer and define the generating function
\begin{equation}\label{equ:genFs}
    G_d(x_1,\dots,x_d)=\sum_{s_1,\dots,s_d\in \N,s_1>1} x_1^{s_1-1}\cdots x_d^{s_d-1}\zeta(s_1,\dots,s_d).
\end{equation}
It is well-known that 
$$G_1(x) = -\gamma-\psi(1-x)$$
where $\gamma$ is Euler's constant and $\psi(x)$ is the digamma function, i.e., the
logarithmic derivative of the gamma function. 
In \cite{GKZ2006}, Gangl, Kaneko and Zagier used the double shuffle relations of \eqref{equ:2RiemStuffle2} and \eqref{equ:2RiemZetaProd} to derive the following equation:
\begin{equation}\label{equ:G2id}
G_2(x+y,x)+G_2(x+y,y)-G_2(x,y)-G_2(y,x)= \frac{G_1(x)-G_1(y)}{x-y},
\end{equation}
and proved some families of MZV identities. Machide \cite{Machide2012} generalized this
to depth three case using the extended (also called regularized) double shuffle relations.

It is well-known that in order to get complete linear relations between MZVs 
one should consider regularized double shuffle relations. For example,
the weighted sum formula of Ohno and Zudilin \cite{OhnoZu2008} states  that
\begin{equation}\label{equ:OWsum}
\sum_{\substack{j\ge 2,k\ge 1\\ j+k=n}} 2^j \zeta(j,k)=(n+1)\zeta(n).
\end{equation}
Later, Guo and Xie \cite{GuoXi2009} generalized \eqref{equ:OWsum} 
to arbitrary depths using regularized double shuffle relations 
(they in fact also used the sum formula which is another consequence 
of the regularized double shuffle relations \cite{IKZ2006}).

In this paper we shall use the generating functions of MZVs \eqref{equ:genFs} to
reformulate double shuffle relations and derive some new identities of MZVs. Notice
that we do not use the extended double shuffle relations, which makes the computation
a little easier. All the identities obtained this way are therefore finite extended
double shuffle relations in the sense of \cite{IKZ2006}.
For example, we get the following interesting result as a
corollary (see Corollary~\ref{cor:1st=3rdabc}) in depth three:
\begin{multline*}
\sum_{\substack{j\ge 2,k,l\ge 1\\ j+k+l=n}}
       2^{j-1} \zeta(j,k,l)
+\sum_{\substack{j\ge 2,k\ge 1\\ j+k=n-1}} 2^j  \zeta(j,k,1)\\
=  n\zeta(n-1,1) +3\zeta(n-2,2)+\zeta(2,n-2)+2\zeta(n)   .
\end{multline*}
for every positive integer$n\ge 2$. For a new result in depth four please see Corollary~\ref{cor:depth4}.

\medskip

\noindent
{\bf Acknowledgement.} Both authors would like to thank the Morningside Center of Mathematics, 
Chinese Academy of Science in Beijing, China for hospitality when the paper was prepared. 
They are indebted to the anonymous referee for his/her very careful reading of the first draft and 
many insightful remarks which improved the paper greatly. 
HY is partially supported by summer research grant from York College of Pennsylvania 
and JZ is partially supported by NSF DMS1162116.

\section{Depth 2: some new identities}
In this section we will derive some new identities of double zeta values using the generating function $G_2$. To begin with, we recall the famous sum formula
essentially known to Euler \cite{Euler1775}:
\begin{equation}\label{equ:EulerSumFormula}
\sum_{\substack{j\ge 2,k\ge 1\\ j+k=n}} \zeta(j,k)= \zeta(n).
\end{equation}
Using generating functions A.\ Granville \cite{Granville1997b} and D.\ Zagier
proved the following generalization to arbitrary depth first conjectured by Moen (see \cite{Hoffman1992}):
\begin{equation}\label{equ:sumConjMZV}
    \sum_{k_1\ge 2, k_2,\dots,k_d\ge 1, k_1+\cdots+k_d=w} \zeta(k_1,k_2,\dots,k_d)=\zeta(w).
\end{equation}

Our first result provides a weighted sum formula similar to but different from \eqref{equ:OWsum}.
\begin{thm}\label{thm:G2Der}
Let $n$ be a positive integer. Then
\begin{align}\label{equ:G2Der}
    \sum_{k=2}^{n-1} k\gz(k,n-k)=&\gz(2,n-2)+2\gz(n)-(n-2)\gz(n-1,1),\quad \forall n\ge 3,\\
    \sum_{k=2}^{n-1} k^2\gz(k,n-k)=&3\gz(2,n-2)+2\gz(3,n-3)+6\gz(n) \notag \\
        & -(2n-6)\gz(n-2,2)-n(n-2)\gz(n-1,1),\quad \forall n\ge 4.  \label{equ:G2DerDer}
\end{align}
\end{thm}
\begin{proof}
Making the substitutions $x\to xt$ and $y\to yt$ in \eqref{equ:G2id}
and comparing the coefficients of $t^{n-2}$ we get
\begin{equation*}
\sum_{k=2}^{n-1} \Big[(x+y)^{k-1} x^{j-1}
+(x+y)^{k-1} y^{j-1}-x^{k-1} y^{j-1}-y^{k-1} x^{j-1}\Big]\gz(k,j)=
    \left(\frac{x^{n-1}-y^{n-1}}{x-y}\right)\gz(n),
\end{equation*}
where $j=n-k$. Differentiating this equation with respect to $x$ we have
\begin{multline}\label{equ:afterDer}
\sum_{k=2}^{n-1} \Big[(k-1)(x+y)^{k-2} x^{j-1}+(j-1)(x+y)^{k-1} x^{j-2}
+(k-1) (x+y)^{k-2} y^{j-1}\\
-(k-1) x^{k-2} y^{j-1}-(j-1)y^{k-1} x^{j-2}\Big]\gz(k,j)
= \left(\frac{(n-1)x^{n-2}}{x-y}-\frac{x^{n-1}-y^{n-1}}{(x-y)^2}\right)\gz(n).
\end{multline}
Specializing to $(x,y)=(0,1)$ we find easily that
$$   (n-2)\gz(n-1,1)+ \sum_{k=2}^{n-1} (k-1)\gz(k,n-k)-\gz(2,n-2)=\gz(n).$$
So \eqref{equ:G2Der} follows from the sum formula \eqref{equ:EulerSumFormula}.

Now multiplying \eqref{equ:afterDer} by $x+y$, differentiating with respect to $x$,
and then specializing to $(x,y)=(0,1)$ we get
\begin{multline*}
    (n-2)^2 \gz(n-1,1)+(2n-6)\gz(n-2,2)+ \sum_{k=2}^{n-1} (k-1)^2\gz(k,n-k)\\
    -\gz(2,n-2)-2\gz(3,n-3)=3\gz(n).
\end{multline*}
Hence \eqref{equ:G2DerDer} quickly follows from \eqref{equ:G2Der} and
the sum formula \eqref{equ:EulerSumFormula}. This completes the proof of the theorem.
\end{proof}

\begin{rem}
It is conceivable that for every fixed positive integer $d$ a compact formula of
$\sum_{k=2}^{n-1} k^d\gz(k,n-k)$ can be obtained by differentiating \eqref{equ:afterDer}
repeatedly, similar to what we have done in Theorem~\ref{thm:G2Der}.
However, it seems to be a difficult problem to find a general formula for all $d$.
Guo, Lei and the second author recently have made some progress along this direction, see
\cite{GuoLeiZhao2013}.
\end{rem}

\section{Depth 3: product of three Riemann zeta values}
In the following we use three different methods to compute the generating function of
the product of three Riemann zeta values:
\begin{equation*}
    \sum_{s_1,s_2,s_3\ge 2} x_1^{s_1-1}x_2^{s_2-1}x_3^{s_3-1}\zeta(s_1) \zeta(s_2) \zeta(s_3).
\end{equation*}

\subsection{First method.}
Combining \eqref{equ:2RiemZetaProd} and \eqref{equ:2Riem1RiemProd} we have
{\allowdisplaybreaks
\begin{align}
\zeta(s_1)\zeta(s_2)\zeta(s_3)
 =& \sum_{\substack{t_1\ge 2,t_2\ge 1\\ t_1+t_2=s_1+s_2}}
\bigg[\binom{t_1-1}{s_1-1}+\binom{t_2-1}{s_2-1}\bigg] \cdot
\bigg\{\sum_{\substack{r_1\ge 2,r_2\ge 1\\ r_1+r_2=t_1+s_3}}
\hspace{-.4cm}\binom{r_1-1}{s_3-1} \zeta(r_1,r_2,t_2)  \notag \\
& +\sum_{\substack{r_1\ge 2,r_2,r_3\ge 1\\ r_1+r_2+r_3 \\ =t_1+t_2+s_3 }}
\hspace{-.2cm}
\binom{r_1-1}{t_1-1} \bigg[\binom{r_2-1}{t_2-r_3}  +
\binom{r_2-1}{t_2-1} \bigg]
\zeta(r_1,r_2,r_3)\bigg\}   \notag\\
=&\sum_{\substack{t_1\ge 2,t_2\ge 1\\ t_1+t_2=s_1+s_2}}\sum_{\substack{r_1\ge 2,r_2\ge 1\\ r_1+r_2=t_1+s_3}}
\binom{t_1-1}{s_1-1}\binom{r_1-1}{s_3-1}\zeta(r_1,r_2,t_2)  \label{equ:3Riem3Prod1}\\
+&\sum_{\substack{t_1\ge 2,t_2\ge 1\\ t_1+t_2=s_1+s_2}}\sum_{\substack{r_1\ge 2,r_2\ge 1\\ r_1+r_2=t_1+s_3}}
\binom{t_2-1}{s_2-1}\binom{r_1-1}{s_3-1}\zeta(r_1,r_2,t_2)   \label{equ:3Riem3Prod2}\\
+&\sum_{\substack{t_1\ge 2,t_2\ge 1\\ t_1+t_2=s_1+s_2}}
\sum_{\substack{r_1\ge 2,r_2,r_3\ge 1\\ r_1+r_2+r_3 \\ =t_1+t_2+s_3 }}
\binom{t_1-1}{s_1-1}\binom{r_1-1}{t_1-1}
\binom{r_2-1}{t_2-r_3}
\zeta(r_1,r_2,r_3) \label{equ:3Riem3Prod3}\\
+&\sum_{\substack{t_1\ge 2,t_2\ge 1\\ t_1+t_2=s_1+s_2}}
\sum_{\substack{r_1\ge 2,r_2,r_3\ge 1\\ r_1+r_2+r_3 \\ =t_1+t_2+s_3 }}
\binom{t_2-1}{s_2-1}\binom{r_1-1}{t_1-1}
\binom{r_2-1}{t_2-r_3}
\zeta(r_1,r_2,r_3) \label{equ:3Riem3Prod4}\\
+&\sum_{\substack{t_1\ge 2,t_2\ge 1\\ t_1+t_2=s_1+s_2}}
\sum_{\substack{r_1\ge 2,r_2,r_3\ge 1\\ r_1+r_2+r_3 \\ =t_1+t_2+s_3 }}
\binom{t_1-1}{s_1-1}\binom{r_1-1}{t_1-1}
\binom{r_2-1}{t_2-1}
\zeta(r_1,r_2,r_3) \label{equ:3Riem3Prod5}\\
+&\sum_{\substack{t_1\ge 2,t_2\ge 1\\ t_1+t_2=s_1+s_2}}
\sum_{\substack{r_1\ge 2,r_2,r_3\ge 1\\ r_1+r_2+r_3 \\ =t_1+t_2+s_3 }}
\binom{t_2-1}{s_2-1}\binom{r_1-1}{t_1-1}
\binom{r_2-1}{t_2-1}
\zeta(r_1,r_2,r_3) \label{equ:3Riem3Prod6}
\end{align}
}
We first treat \eqref{equ:3Riem3Prod1} to \eqref{equ:3Riem3Prod6} using the binomial identities repeatedly to derive formulas involving the generating functions $G_3$. To save space we only compute \eqref{equ:3Riem3Prod1} in details and leave the others to the interested reader. Also we will use the shorthand $x_{ij}=x_j+x_j$ and $x_{ijk}=x_j+x_j+x_k$ in what follows. We also
use the shorthand $\sum$\eqref{equ:3Riem3Prod1} to stand for $\sum_{\substack{s_1,s_2,s_3\ge 1}} x_1^{s_1-1} x_2^{s_2-1} x_3^{s_3-1}$ \eqref{equ:3Riem3Prod1}. Now
{\allowdisplaybreaks
\begin{align*}
\ &\sum \eqref{equ:3Riem3Prod1}
=\sum_{\substack{s_1,s_2,s_3\ge 2}}
   \sum_{\substack{t_1\ge 2,t_2\ge 1\\ t_1+t_2=s_1+s_2}}
    \sum_{\substack{r_1\ge 2,r_2\ge 1\\ r_1+r_2=t_1+s_3}}
\binom{t_1-1}{s_1-1}\binom{r_1-1}{s_3-1}
x_1^{s_1-1}x_2^{s_2-1}x_3^{s_3-1}\zeta(r_1,r_2,t_2) \\
=&\sum_{\substack{s_1,s_2,s_3\ge 2}}
   \sum_{\substack{t_1\ge 2,t_2\ge 1\\ t_1+t_2=s_1+s_2}}
    \sum_{\substack{r_1\ge 2,r_2\ge 1\\ r_1+r_2=t_1+s_3}}
\binom{t_1-1}{s_1-1}\binom{r_1-1}{s_3-1}
x_1^{s_1-1}x_2^{t_1-s_1}x_2^{t_2-1}x_3^{s_3-1}\zeta(r_1,r_2,t_2) \\
=&    \sum_{\substack{t_1,s_3,r_1\ge 2,t_2,r_2\ge 1\\  r_1+r_2=t_1+s_3}}
\Big[(x_{12})^{t_1-1}-x_2^{t_1-1}\Big]x_3^{s_3-1}\binom{r_1-1}{s_3-1}
x_2^{t_2-1}\zeta(r_1,r_2,t_2) \\
&   \hskip 2cm -\sum_{\substack{t_1,s_3,r_1\ge 2,r_2\ge 1\\  r_1+r_2=t_1+s_3}}
x_1^{t_1-1}x_3^{s_3-1}\binom{r_1-1}{s_3-1}
\zeta(r_1,r_2,1) \\
=&   \sum_{\substack{t_1,s_3,r_1\ge 2,t_2,r_2\ge 1\\  r_1+r_2=t_1+s_3}}
\Big[(x_{12})^{r_2-1}(x_{12})^{r_1-s_3}-x_2^{r_2-1}x_2^{r_1-s_3}\Big]x_3^{s_3-1}\binom{r_1-1}{s_3-1}
x_2^{t_2-1}\zeta(r_1,r_2,t_2) \\
&   \hskip 2cm -\sum_{\substack{t_1,s_3,r_1\ge 2,r_2\ge 1\\  r_1+r_2=t_1+s_3}}
x_1^{r_2-1}x_1^{r_1-s_3}x_3^{s_3-1}\binom{r_1-1}{s_3-1}
\zeta(r_1,r_2,1) \\
=& \sum_{r_1\ge 1,t_2,r_2\ge 1}
\Big\{(x_{12})^{r_2-1}\big[(x_{123})^{r_1}-(x_{12})^{r_1}\big]
-x_2^{r_2-1}\big[(x_{23})^{r_1}-x_2^{r_1}\big]\Big\} x_2^{t_2-1}\zeta(r_1+1,r_2,t_2)  \\
& -\sum_{\substack{t_1,s_3,r_1\ge 2,r_2\ge 1\\  r_1+r_2=t_1+s_3}}
 x_1^{r_2-1}\big[(x_{13})^{r_1-1}-x_1^{r_1-1}\big]\zeta(r_1,r_2,1)
    +\sum_{r_1\ge 2} x_3^{r_1-1}\zeta(r_1,1,1)  \\
=&   G_3(x_{123},x_{12}, x_2)- G_3(x_{12},x_{12}, x_2)- G_3(x_{23}, x_2, x_2)+ G_3(x_2 , x_2, x_2)\\
&   \hskip 2cm -\sum_{\substack{t_1,s_3,r_1\ge 2,r_2\ge 1\\  r_1+r_2=t_1+s_3}}
x_1^{r_2-1}\big[(x_{13})^{r_1-1}-x_1^{r_1-1}\big] \zeta(r_1,r_2,1)
+\sum_{r_1\ge 2} x_3^{r_1-1}\zeta(r_1,1,1)
\end{align*}
}
Similarly we find
\begin{multline*}
\sum \eqref{equ:3Riem3Prod3} 
=G_3(x_{123},x_{23}, x_2)- G_3(x_{23},x_{23}, x_2)- G_3(x_{12}, x_2, x_2)+ G_3(x_2 , x_2, x_2)\\
 -\sum_{r_1\ge 2,r_2\ge 1} \big[(x_{13})^{r_1-1}-x_3^{r_1-1}\big]x_3^{r_2-1}\zeta(r_1,r_2,1)
+\sum_{r_1\ge 2}x_1^{r_1-1}\zeta(r_1,1,1)
\end{multline*}
which is just $\sum$\eqref{equ:3Riem3Prod1} under the operation $x_1\leftrightarrow x_3$. Further,
\begin{multline*}
\sum\eqref{equ:3Riem3Prod5} 
=G_3(x_{123},x_{23}, x_3)- G_3(x_{23},x_{23}, x_2)- G_3(x_{13}, x_3, x_3)+ G_3(x_3 , x_3, x_3)\\
 -\sum_{r_1\ge 2,r_2\ge 1} \big[(x_{12})^{r_1-1}-x_2^{r_1-1}\big]x_2^{r_2-1}\zeta(r_1,r_2,1) +\sum_{r_1\ge 2} x_1^{r_1-1} \zeta(r_1,1,1)
\end{multline*}
which is $\sum$\eqref{equ:3Riem3Prod1} under the operation ${\rm Cyc}( x_1\to x_2 \to x_3 \to x_1).$
By the same argument we easily find that $\sum$\eqref{equ:3Riem3Prod2}, $\sum$\eqref{equ:3Riem3Prod4}
and $\sum$\eqref{equ:3Riem3Prod6} can all be obtained from $\sum$\eqref{equ:3Riem3Prod1} under different
permutations of $x_1, x_2$ and $x_3.$  Therefore
\begin{multline*}
    \sum_{s_1,s_2,s_3\ge 2} x_1^{s_1-1}x_2^{s_2-1}x_3^{s_3-1}\zeta(s_1) \zeta(s_2) \zeta(s_3)\\
=\bigoplus_{\calS(x_1,x_2, x_3)}\Big\{
    G_3(x_{123},x_{12}, x_2)- G_3(x_{12},x_{12}, x_2)- G_3(x_{23}, x_2, x_2)+ G_3(x_2 , x_2, x_2)\\
-\sum_{\substack{t_1,s_3,r_1\ge 2,r_2\ge 1\\  r_1+r_2=t_1+s_3}}
x_1^{r_2-1}\big[(x_{13})^{r_1-1}-x_1^{r_1-1}\big]
\zeta(r_1,r_2,1)
+\sum_{r_1\ge 2} x_3^{r_1-1}\zeta(r_1,1,1) \Big\}
\end{multline*}
where, for a function $f(x_1,\dots,x_k)$ we define
\begin{equation*}
 \bigoplus_{\calS(x_1,\dots,x_k)} f(x_1,\dots,x_k)=\sum_{\gs: \text{ permutations of }1,\dots,k} f(x_{\gs(1)},\dots,x_{\gs(k)}).
\end{equation*}

\subsection{Second method.}
Multiplying \eqref{equ:2RiemZetaProd} by $\zeta(s_3)$ and using the stuffle relations we get
{\allowdisplaybreaks
\begin{align}
\ &\sum_{s_1,s_2,s_3\ge 2} x_1^{s_1-1}x_2^{s_2-1}x_3^{s_3-1}\zeta(s_1) \zeta(s_2) \zeta(s_3)\notag\\
 =& \sum_{s_1,s_2,s_3\ge 2} x_1^{s_1-1}x_2^{s_2-1}x_3^{s_3-1}
 \sum_{\substack{t_1\ge 2,t_2\ge 1\\ t_1+t_2=s_1+s_2}}
\bigg[\binom{t_1-1}{s_1-1}+\binom{t_1-1}{s_2-1}\bigg]\zeta(t_1,t_2,s_3)\label{equ:3product1}\\
+&\sum_{s_1,s_2,s_3\ge 2} x_1^{s_1-1}x_2^{s_2-1}x_3^{s_3-1}
\sum_{\substack{t_1\ge 2,t_2\ge 1\\ t_1+t_2=s_1+s_2}}
\bigg[\binom{t_1-1}{s_1-1}+\binom{t_1-1}{s_2-1}\bigg]\zeta(t_1,s_3,t_2)\label{equ:3product2}\\
+&\sum_{s_1,s_2,s_3\ge 2} x_1^{s_1-1}x_2^{s_2-1}x_3^{s_3-1}
\sum_{\substack{t_1\ge 2,t_2\ge 1\\ t_1+t_2=s_1+s_2}}
\bigg[\binom{t_1-1}{s_1-1}+\binom{t_1-1}{s_2-1}\bigg]\zeta(s_3,t_1,t_2)\label{equ:3product3}\\
+&\sum_{s_1,s_2,s_3\ge 2} x_1^{s_1-1}x_2^{s_2-1}x_3^{s_3-1}
\sum_{\substack{t_1\ge 2,t_2\ge 1\\ t_1+t_2=s_1+s_2}}
\bigg[\binom{t_1-1}{s_1-1}+\binom{t_1-1}{s_2-1}\bigg]\zeta(t_1,t_2+s_3)\label{equ:3product4}\\
+&\sum_{s_1,s_2,s_3\ge 2} x_1^{s_1-1}x_2^{s_2-1}x_3^{s_3-1}
\sum_{\substack{t_1\ge 2,t_2\ge 1\\ t_1+t_2=s_1+s_2}}
\bigg[\binom{t_1-1}{s_1-1}+\binom{t_1-1}{s_2-1}\bigg]\zeta(t_1+s_3,t_2). \label{equ:3product5}
\end{align}
}
Using the techniques similar to the one used in the proceeding subsection we get
{\allowdisplaybreaks
\begin{align*}
\eqref{equ:3product1}
=&\bigoplus_{\calS(x_1,x_2)} \Big\{G_3(x_{12},x_2,x_3)
 -\sum_{t_1\ge 2,s_3\ge 1} x_1^{t_1-1}  x_3^{s_3-1}\zeta(t_1,1,s_3)+\sum_{t_1\ge 2} x_1^{t_1-1} \zeta(t_1,1,1) \\
& -\sum_{t_1\ge 2,t_2\ge 1} (x_{12})^{t_1-1} x_2^{t_2-1} \zeta(t_1,t_2,1)
   -G_3(x_2,x_2,x_3) +\sum_{t_1\ge 2,t_2\ge 1}   x_2^{t_1+t_2-2} \zeta(t_1,t_2,1) \Big\}\\
\eqref{equ:3product2}
=&\bigoplus_{\calS(x_1,x_2)} \Big\{G_3(x_{12},x_3,x_2)
 -\sum_{t_1\ge 2,s_3\ge 1} x_1^{t_1-1}  x_3^{s_3-1}\zeta(t_1,s_3,1)+\sum_{t_1\ge 2} x_1^{t_1-1} \zeta(t_1,1,1) \\
& -\sum_{t_1\ge 2,t_2\ge 1} (x_{12})^{t_1-1} x_2^{t_2-1} \zeta(t_1,1,t_2)
   -G_3(x_2,x_3,x_2) +\sum_{t_1\ge 2,t_2\ge 1}   x_2^{t_1+t_2-2} \zeta(t_1,1,t_2) \Big\}\\
\eqref{equ:3product3}
=&\bigoplus_{\calS(x_1,x_2)} \Big\{ G_3(x_3,x_{12},x_2)-G_3(x_3,x_2,x_2)\\
& -\sum_{s_3\ge 2,t_1\ge 1} x_1^{t_1-1} x_3^{s_3-1} \zeta(s_3,t_1,1)
    +\sum_{s_3\ge 2}  x_3^{s_3-1}  \zeta(s_3,1,1)  \Big\}\\
\eqref{equ:3product4}
=&\bigoplus_{\calS(x_1,x_2)} \Big\{
 \frac{G_2(x_{12},x_3)-G_2(x_{12},x_2)}{x_3-x_2}-\frac{1}{x_2}G_2(x_{12},x_2)\\
& -\frac{G_2(x_2,x_3)-G_2(x_2,x_2)}{x_3-x_2}
    +\frac{1}{x_2}G_2(x_2,x_2)-\frac{1}{x_3}G_2(x_1,x_3) \\
& +\frac{1}{x_2}\sum_{t_1\ge 2}\Big((x_{12})^{t_1-1}-x_2^{t_1-1}\Big)\zeta(t_1,1)
   +\frac{1}{x_3}\sum_{t_1\ge 1} x_1^{t_1-1}\Big(\zeta(t_1,1)+x_3\zeta(t_1,2)\Big) \Big\}\\
\eqref{equ:3product5}
=&\bigoplus_{\calS(x_1,x_2)} \Big\{ \frac{G_2(x_3,x_2)-G_2(x_{12},x_2)}{x_3-x_{12}}
 -\frac{G_2(x_3,x_2)-G_2(x_2,x_2)}{x_3-x_2}+\frac{1}{x_2}G_2(x_2,x_2)\\
&   -\frac{1}{x_{12}}G_2(x_{12},x_2)-\sum_{t\ge 2}\left(\frac{x_3^{t-1}-x_1^{t-1}}{x_3-x_1}-x_3^{t-2}-x_1^{t-2} \right) \zeta(t,1)
-\zeta(2,1)\Big\}
\end{align*}
}
Therefore
{\allowdisplaybreaks
\begin{align*}
&\sum_{s_1,s_2,s_3\ge 2}x_1^{s_1-1}x_2^{s_2-1}x_3^{s_3-1}\zeta(s_1)\zeta(s_2)\zeta(s_3)\\
=& \bigoplus_{\calS(x_1,x_2)} \Big\{G_3(x_{12},x_2,x_3)
 -\sum_{t_1\ge 2,t_2\ge 1} x_1^{t_1-1}  x_3^{t_2-1}\zeta(t_1,1,t_2)+2\sum_{t_1\ge 2} x_1^{t_1-1} \zeta(t_1,1,1) \\
 -&\sum_{t_1\ge 2,t_2\ge 1} (x_{12})^{t_1-1} x_2^{t_2-1} \zeta(t_1,t_2,1)
   -\sum_{t_1\ge 2,t_2\ge 1} x_3^{t_1-1} x_1^{t_2-1} \zeta(t_1,t_2,1)
   +\sum_{t_1\ge 2,t_2\ge 1}   x_2^{t_1+t_2-2} \zeta(t_1,t_2,1) \\
 +&\sum_{t_1\ge 2,t_2\ge 1}   x_2^{t_1+t_2-2} \zeta(t_1,1,t_2)
 -\sum_{t_1\ge 2,t_2\ge 1} x_1^{t_1-1}  x_3^{t_2-1}\zeta(t_1,t_2,1)
    -\sum_{t_1\ge 2,t_2\ge 1} (x_{12})^{t_1-1} x_2^{t_2-1} \zeta(t_1,1,t_2)
\\
+& G_3(x_{12},x_3,x_2)-G_3(x_2,x_3,x_2)
+ G_3(x_3,x_{12},x_2)-G_3(x_3,x_2,x_2)-G_3(x_2,x_2,x_3)  \\
+&\frac{G_2(x_{12},x_3)-G_2(x_{12},x_2)}{x_3-x_2}
    -\frac{1}{x_2}G_2(x_{12},x_2)
   -\frac{G_2(x_2,x_3)-G_2(x_2,x_2)}{x_3-x_2}\\
+&\frac{2}{x_2}G_2(x_2,x_2)-\frac{1}{x_3}G_2(x_1,x_3)
    +\frac{G_2(x_3,x_2)-G_2(x_{12},x_2)}{x_3-x_2-x_1}-\frac{1}{x_{12}}G_2(x_{12},x_2)\\
+&\frac{1}{x_2}\sum_{t_1\ge 2}(x_{12})^{t_1-1}\zeta(t_1,1)
   +\frac{1}{x_3}\sum_{t_1\ge 2} x_1^{t_1-1}\Big(\zeta(t_1,1)+x_3\zeta(t_1,2)\Big)    +\sum_{t_1\ge 2}  x_3^{t_1-1}  \zeta(t_1,1,1)  \\
-&\frac{G_2(x_3,x_2)-G_2(x_2,x_2)}{x_3-x_2}
-\sum_{t\ge 2}\left(\frac{x_3^{t-1}-x_1^{t-1}}{x_3-x_1}-x_3^{t-2} \right) \zeta(t,1)
-\zeta(2,1)\Big\}
\end{align*}
}

\subsection{Third method.}
Repeated use of stuffle relations yields
\begin{multline}\label{equ:4types}
 \zeta(s_1)\zeta(s_2)\zeta(s_3)=\zeta(s_1,s_2,s_3)
 +\zeta(s_1,s_3,s_2)+\zeta(s_3,s_1,s_2)
 +\zeta(s_1,s_2+s_3)\\
 +\zeta(s_1+s_3,s_2)
 +\zeta(s_2,s_1,s_3)
 +\zeta(s_2,s_3,s_1)+\zeta(s_3,s_2,s_1)\\
 +\zeta(s_2,s_1+s_3)
 +\zeta(s_2+s_3,s_1)
 +\zeta(s_1+s_2,s_3) +\zeta(s_3,s_1+s_2) +\zeta(s_1+s_2+s_3).
\end{multline}
On the right hand side of the above there are essentially four types of MZVs.
Similar computation as above leads to the following four expressions
of their generating functions:
{\allowdisplaybreaks
\begin{align*}
&\sum_{s_1,s_2,s_3\ge 2} x_1^{s_1-1}x_2^{s_2-1}x_3^{s_3-1}\zeta(s_1,s_2,s_3)=G_3(x_1,x_2,x_3)
-\sum_{s_1\ge 2,s_2\ge 1} x_1^{s_1-1}x_2^{s_2-1}\zeta(s_1,s_2,1)\\
& -\sum_{s_1\ge 2,s_3\ge 1} x_1^{s_1-1}x_3^{s_3-1}\zeta(s_1,1,s_3)
+\sum_{s_1\ge 2} x_1^{s_1-1}\zeta(s_1,1,1)\\
& \sum_{s_1,s_2,s_3\ge 2} x_1^{s_1-1}x_2^{s_2-1}x_3^{s_3-1}
\zeta(s_1,s_2+s_3)=\frac{G_2(x_1,x_3)-G_2(x_1,x_2)}{x_3-x_2}-\frac{G_2(x_1,x_2)}{x_2}\\
& -\frac{G_2(x_1,x_3)}{x_3}+\sum_{s_1\ge 2} x_1^{s_1-1}\zeta(s_1,2)
+\left(\frac{1}{x_3}+\frac{1}{x_2}\right)\sum_{s_1\ge 2} x_1^{s_1-1}\zeta(s_1,1)\\
&\sum_{s_1,s_2,s_3\ge 2} x_1^{s_1-1}x_2^{s_2-1}x_3^{s_3-1}
\zeta(s_2+s_3,s_1)
=\frac{G_2(x_3,x_1)-G_2(x_2,x_1)}{x_3-x_2}-\frac{1}{x_2}G_2(x_2,x_1)\\
&-\frac{1}{x_3}G_2(x_3,x_1)+\sum_{s_1\ge 1} x_1^{s_1-1}\zeta(2,s_1)
-\sum_{s\ge 2}\left(\frac{x_3^{s-1}-x_2^{s-1}}{x_3-x_2}
    -x_3^{s-2}-x_2^{s-2}+\gd_{s,2}\right)\zeta(s,1)\\
&\sum_{s_1,s_2,s_3\ge 2} x_1^{s_1-1}x_2^{s_2-1}x_3^{s_3-1}\zeta(s_1+s_2+s_3)
=\frac{1}{x_3}\left(\frac{G_1(x_3)}{x_3}+\frac{G_1(x_1)}{x_1}
    -\frac{G_1(x_3)-G_1(x_1)}{x_3-x_1} \right)\\
&+\frac{1}{x_3-x_2}\left(\frac{G_1(x_3)-G_1(x_1)}{x_3-x_1}
    -\frac{G_1(x_3)}{x_3}-\frac{G_1(x_2)-G_1(x_1)}{x_2-x_1}
    +\frac{G_1(x_2)}{x_2}\right)--\frac{\zeta(2)}{x_3}\\
&-\frac{1}{x_2}\left(\frac{G_1(x_2)-G_1(x_1)}{x_2-x_1}
    -\frac{G_1(x_2)}{x_2}-\frac{G_1(x_1)}{x_1}+\zeta(2)\right)
        +\frac{G_1(x_1)}{x_1^2}-\zeta(3)-\frac{\zeta(2)}{x_1}
\end{align*}}
Therefore \eqref{equ:4types} becomes
\begin{multline*}
    \sum_{s_1,s_2,s_3\ge 2} x_1^{s_1-1}x_2^{s_2-1}x_3^{s_3-1}\zeta(s_1) \zeta(s_2) \zeta(s_3)
=\bigoplus_{\calS(x_1,x_2, x_3)}\Big\{G_3(x_1,x_2,x_3)-\frac{G_2(x_1,x_2)}{x_2}\\
-\frac{G_2(x_2,x_1)}{x_2}+\frac{G_2(x_1,x_3)+G_2(x_3,x_1)}{x_3-x_2}
-\sum_{s_1\ge 2,s_2\ge 1} x_1^{s_1-1}x_2^{s_2-1}\zeta(s_1,s_2,1)
+\sum_{s\ge 2}x_1^{s-2}\zeta(s,1)\\
+\sum_{s\ge 2} x_1^{s-1}\zeta(s,1,1)
-\sum_{s_1\ge 2,s_3\ge 1} x_1^{s_1-1}x_3^{s_3-1}\zeta(s_1,1,s_3)
+\frac{1}{x_2}\sum_{s\ge 2} x_1^{s-1}\zeta(s,1)\Big\}-\zeta(3)\\
+\bigoplus_{\calC(x_1,x_2, x_3)}\Big\{
\sum_{s\ge 2} x_1^{s-1}\zeta(s,2)+\sum_{s\ge 2} x_1^{s-1}\zeta(2,s)
-\sum_{s\ge 2}\left(\frac{x_3^{s-1}-x_2^{s-1}}{x_3-x_2}\right)\zeta(s,1)
+\frac{G_1(x_1)}{x_1^2}-\frac{\zeta(2)}{x_1}\Big\}\\
+\bigoplus_{\calC(x_2, x_3)}\Big\{
\frac{1}{x_3-x_2}\left(\frac{G_1(x_3)-G_1(x_1)}{x_3-x_1}
    +\frac{G_1(x_2)}{x_2}\right)
-\frac{1}{x_2}\left(\frac{G_1(x_2)-G_1(x_1)}{x_2-x_1}-\frac{G_1(x_1)}{x_1}\right)\Big\}
\end{multline*}
Here, for a function $f(x_1,\dots,x_k)$ we define
\begin{equation*}
 \bigoplus_{\calC(x_1,\dots,x_k)} f(x_1,\dots,x_k)=\sum_{i=1}^k f(x_{i},x_{i+1},\dots,x_{i+k-1})
\end{equation*}
where the subscript is taken modulo $k$.

By comparing the first and the third method we get
\begin{thm}\label{thm:1st=3rd}
We have
{\allowdisplaybreaks
\begin{multline*}
\bigoplus_{\calS(x_1,x_2, x_3)}\Big\{
    G_3(x_{123},x_{12}, x_2)- G_3(x_{12},x_{12}, x_2)- G_3(x_{23}, x_2, x_2)\\
+ G_3(x_2 , x_2, x_2)-G_3(x_1,x_2,x_3)-\frac{G_2(x_1,x_3)+G_2(x_3,x_1)}{x_3-x_2}+\frac{G_2(x_1,x_2)}{x_2}
+\frac{G_2(x_2,x_1)}{x_2}\Big\}\\
-\bigoplus_{\calC(x_2, x_3)}\Big\{
\frac{1}{x_3-x_2}\left(\frac{G_1(x_3)-G_1(x_1)}{x_3-x_1}
    +\frac{G_1(x_2)}{x_2}\right)
-\frac{G_1(x_2)-G_1(x_1)}{x_2(x_2-x_1)}+\frac{G_1(x_1)}{x_1x_2}
\Big\} \\
=\bigoplus_{\calS(x_1,x_2, x_3)}\Big\{\sum_{r_1\ge 2,r_2\ge 1}
x_1^{r_2-1}\big[(x_{13})^{r_1-1}-x_1^{r_1-1}-x_3^{r_1-1}\big]
\zeta(r_1,r_2,1)
+\sum_{s\ge 2}x_1^{s-2}\zeta(s,1)\\
-\sum_{s_1\ge 2,s_3\ge 1} x_1^{s_1-1}x_3^{s_3-1}\zeta(s_1,1,s_3)
+\frac{1}{x_2}\sum_{s\ge 2} x_1^{s-1}\zeta(s,1)
-\sum_{s\ge 2}\left(\frac{x_3^{s-1}}{x_3-x_2}\right)\zeta(s,1)
\Big\}\\
+\bigoplus_{\calC(x_1,x_2, x_3)}\Big\{
\sum_{s\ge 2} x_1^{s-1}\zeta(s,2)+\sum_{s\ge 2} x_1^{s-1}\zeta(2,s)
+\frac{G_1(x_1)}{x_1^2}-\frac{\zeta(2)}{x_1}\Big\}-\zeta(3)\\
\end{multline*}
}
\end{thm}
\begin{proof}
Clear.
\end{proof}

This theorem is equivalent to the following result which can be regarded as
a parametric family of weighted sum formulas.
\begin{thm}\label{thm:1st=3rdabc}
Let $a$, $b$ and $c$ be any real numbers and let $\gs=a+b$. Then for any positive
integer $n\ge 4$ we have
{\allowdisplaybreaks
\begin{multline*}
\bigoplus_{\calS(a,b,c)}\bigg\{
\sum_{\substack{j\ge 2,k,l\ge 1\\ j+k+l=n}}
     \Big(ac(a+b+c)^{j-1}(a+b)^{k-1}b^l-ac(a+b)^{j+k-2}b^l-ac(b+c)^{j-1}b^{k+l-1} \\
+ac b^{j+k+l-2}-a^j b^k c^l \Big) \zeta(j,k,l)
+\sum_{\substack{j\ge 2,k\ge 1\\ j+k=n}}
\Big( c a^j b^{k-1}+c a^k b^{j-1}-\frac{b(a^j c^k+a^k c^j)}{c-b}\Big)\zeta(j,k)\bigg\} \\
-\bigoplus_{\calC(b,c)}\Big\{
\frac{1}{c-b}\left(\frac{ab c^n-bc a^n}{c-a}
    +acb^{n-1}\right)
-\frac{ac(b^{n-1}-a^{n-1})}{b-a}+ca^{n-1}\Big\}\zeta(n) \\
=\bigoplus_{\calS(a,b,c)}\Big\{
\sum_{\substack{j\ge 2,k\ge 1\\ j+k=n-1}}
   \Big( a^k bc\big[(a+c)^{j-1}-a^{j-1}-c^{j-1}\big] \zeta(j,k,1)
-  a^j b c^k\zeta(j,1,k) \Big)\\
+\Big(a^{n-1} c + a^{n-2}bc - \frac{ab c^{n-1}}{c-b}\Big)\zeta(n-1,1)
+\frac12 a^{n-2}bc\big(\zeta(n-2,2)+\zeta(2,n-2)\big)+\frac12 a^{n-2}bc\zeta(n)
\Big\}.
\end{multline*}
}
\end{thm}
\begin{proof}
In Theorem~\ref{thm:1st=3rd} we first set $x_1=at$, $x_2=bt$ and $x_3=ct$. Then by comparing
the coefficient of $t^{n-3}$ we get Theorem~\ref{thm:1st=3rdabc}.
\end{proof}

The following weighted sum formula seems to be new.
\begin{cor}\label{cor:1st=3rdabc}
For any positive integer $n\ge 2$ we have
\begin{multline*}
\sum_{\substack{j\ge 2,k,l\ge 1\\ j+k+l=n}}
       2^{j-1} \zeta(j,k,l)
+\sum_{\substack{j\ge 2,k\ge 1\\ j+k=n-1}} 2^j  \zeta(j,k,1)\\
=  n\zeta(n-1,1) +3\zeta(n-2,2)+\zeta(2,n-2)+2\zeta(n)   .
\end{multline*}
\end{cor}
\begin{proof}
Setting $a=1$ and letting $b\to 1$ and then $c\to 1$  in Theorem \ref{thm:1st=3rdabc} we get
\begin{multline*}
\sum_{\substack{j\ge 2,k,l\ge 1\\ j+k+l=n}}
     6\Big(3^{j-1}2^{k-1}-2^{j+k-2}-2^{j-1} \Big) \zeta(j,k,l)
+\sum_{\substack{j\ge 2,k\ge 1\\ j+k=n}}
3 (6-n )\zeta(j,k)
-\frac{n^2-9n+20}{2} \zeta(n) \\
=\sum_{\substack{j\ge 2,k\ge 1\\ j+k=n-1}}
   6\Big( (2^{j-1}-2) \zeta(j,k,1)-\zeta(j,1,k) \Big)
+3 (6-n)\zeta(n-1,1) +3\zeta(n-2,2)+3\zeta(2,n-2)   .
\end{multline*}
The following two special type sum formulas are special cases of \cite[Thm.~2.3]{HoffmanOh2003}
and \cite[Thm.~5.1]{Hoffman1992}, respectively:
\begin{align}
\label{equ:HoffSumFormula2}
  \sum_{\substack{j\ge 2,k\ge 1\\ j+k=n-1}} \zeta(j,1,k)=&\zeta(n-1,1)+\zeta(2,n-2)\\
\label{equ:HoffSumFormula3}
  \sum_{\substack{j\ge 2,k\ge 1\\ j+k=n-1}} \zeta(j,k,1)=&\zeta(n-1,1)+\zeta(n-2,2).
\end{align}
Notice that \eqref{equ:HoffSumFormula3} is a also special case of Eie's generalized sum formula in \cite{ELO2009}. Combining with sum formula \eqref{equ:EulerSumFormula} we get
\begin{multline*}
\sum_{\substack{j\ge 2,k,l\ge 1\\ j+k+l=n}}
     6\Big(3^{j-1}2^{k-1}-2^{j+k-2}-2^{j-1} \Big) \zeta(j,k,l)
-\frac{n^2-3n-16}{2} \zeta(n) \\
=3 \sum_{\substack{j\ge 2,k\ge 1\\ j+k=n-1}} 2^j  \zeta(j,k,1)
-3n\zeta(n-1,1) -9\zeta(n-2,2)-3\zeta(2,n-2)   .
\end{multline*}
Dividing by 3 throughout we get
\begin{multline*}
\sum_{\substack{j\ge 2,k,l\ge 1\\ j+k+l=n}}
      \Big(3^{j-1}2^k-2^{j+k-1}-2^j\Big) \zeta(j,k,l)
-\frac{n^2-3n-16}{6} \zeta(n) \\
= \sum_{\substack{j\ge 2,k\ge 1\\ j+k=n-1}} 2^j  \zeta(j,k,1)
-n\zeta(n-1,1) -3\zeta(n-2,2)-\zeta(2,n-2)   .
\end{multline*}
Hence the corollary follows from \cite[Cor.\ 4.1]{Machide2012}.
\end{proof}

By comparing the second and the third method we can get another identity
involving the generating function $G_3$. However, it is quite
long so here we just write down the following version concerning
a parametric family of weighted sum formulas.

\begin{thm}\label{thm:2nd=3rdabc}
Let $a$, $b$ and $c$ be any real numbers and let $\gs=a+b$. Then for any positive
integer $n\ge 2$ we have
\allowdisplaybreaks
{
\begin{multline*}
\sum_{\substack{j\ge 2,k,l\ge 1\\ j+k+l=n}}
 \bigoplus_{\calC(a,b)} \bigg(a\gs^{j-1}(b^k c^l+b^lc^k)+a\gs^{k-1} b^l c^j\
-\bigoplus_{\calC(j,k,l)}  \Big(a^{j+k-1}bc^l + a^j b^k c^l\Big)\bigg)  \zeta(j,k,l) \\
 -\sum_{\substack{j\ge 2,k\ge 1\\ j+k=n-1}}\bigoplus_{\calC(a,b)}\Big(a \gs^{j-1} b^k c
    -a^{j+k-1}bc-a^j b^k c\Big)\zeta(j,k,1)\\
-\sum_{\substack{j\ge 2,k\ge 1\\ j+k=n-1}}\bigoplus_{\calC(a,b)}
    \Big(a \gs^{j-1} b^k c-a^{j+k-1}bc-a^k b c^j-a^k c b^j\Big) \zeta(j,1,k) \\
= \sum_{\substack{j\ge 2,k\ge 1\\ j+k=n}}\bigg(
\bigoplus_{\calC(a,b,c)}  \bigoplus_{\calC(j,k)} \Big(\frac{a^j(b c^k-c b^k)}{c-b}
-a^j b^{k-1} c-a^j b c^{k-1} \Big) -\bigoplus_{\calC(a,b)}\Big(\frac{\gs^{j-1} a (b c^k- c b^k)}{c-b}-a^j b c^{k-1}\\
 -\gs^{j-1}b^{k-1} ac +\frac{a b^k c^j-\gs^{j-1} a c b^k}{c-b-a}-\gs^{j-2}b^k a c
+ \bigoplus_{\calC(j,k)}  \Big(a c b^{j+k-2} -\frac{b^{k-1} a (b c^j- c b^j)}{c-b}\Big)\Big)\bigg)\zeta(j,k)\\
-\bigoplus_{\calC(a,b)}\Big( a c \gs^{n-2}-a^{n-1}c-c^{n-1}a-2a b^{n-2} c+\frac{a cb^{n-1}-b c a^{n-1}}{2(b-a)}\Big)  \zeta(n-1,1) \\
+\bigoplus_{\calC(b,c)}\bigg(\frac{b^2(a^2 c^{n-1}-a^n c)}{c(c-b)(c-a)}
 +a^{n-1}b \bigg)\zeta(n)+ bca^{n-2}\zeta(n) +abc^{n-2}\zeta(n-2,2) \\
+\bigoplus_{\calC(a,b,c)} a^{n-2}bc \zeta(2,n-2)  -\gd_{n,2}(ab+ac+bc)\zeta(2)-\gd_{n,3}abc(\zeta(2,1)+\zeta(3))
\end{multline*}
}
\end{thm}

\begin{proof}
Similar to that of Theorem~\ref{thm:1st=3rdabc}. We leave the details to the interested reader.
\end{proof}

\begin{cor}\label{cor:3RieZetaProd}
Let $n$ be any positive integer. Then
\begin{multline*}
\sum_{\substack{j\ge 2,k,l\ge 1\\ j+k+l=n}}
\big(2^{j+1}+2^k\big)  \zeta(j,k,l)
-\sum_{\substack{j\ge 2,k\ge 1\\ j+k=n-1}}2^j\big(\zeta(j,1,k)+\zeta(j,k,1)\big)
+\sum_{\substack{j\ge 2,k\ge 1\\ j+k=n}} 2^j \cdot k \cdot \zeta(j,k)\\
=\frac{(n+3)(n+1)}2 \zeta(n)-3\zeta(n-2, 2)-(2^{n-1}+n)\zeta(n-1, 1)-3\zeta(2, n-2).
\end{multline*}
\end{cor}
\begin{proof}
Setting $a=1$ and letting $b,c\to 1$ in Theorem \ref{thm:2nd=3rdabc} we get
\begin{multline*}
\sum_{\substack{j\ge 2,k,l\ge 1\\ j+k+l=n}}
\big(2^{j+1}+2^k-6 \big)  \zeta(j,k,l)
-\sum_{\substack{j\ge 2,k\ge 1\\ j+k=n-1}}\big[(2^j-6 ) \zeta(j,1,k)+(2^j-4) \zeta(j,k,1)\big]\\
+\sum_{\substack{j\ge 2,k\ge 1\\ j+k=n}}\big(2^j(k-1)-2^{j-1}-5n+22\big)\zeta(j,k)\\
=\frac{n^2-9n+32}2 \zeta(n)+\zeta(n-2, 2)-(2^{n-1}+n-10)\zeta(n-1, 1)+3\zeta(2, n-2).
\end{multline*}
Using the weighted sum formula \eqref{equ:OWsum}, the sum formulas \eqref{equ:sumConjMZV}, \eqref{equ:HoffSumFormula2}, and \eqref{equ:HoffSumFormula3} we can derive our corollary quickly.
\end{proof}

\section{Depth 3: product of zeta and double zeta}
Set $\bfx^{\bfs-\bfone}=x_1^{s_1-1}x_2^{s_2-1}x_3^{s_3-1}$ throughout this section.
By \eqref{equ:2Riem1RiemProd} for all $s_1,s_3\ge 2$ and $s_2\ge 1$,
{\allowdisplaybreaks
\begin{align}
 & \sum_{s_1,s_3\ge 2,s_2\ge 1} \bfx^{\bfs-\bfone}\zeta(s_1,s_2)\,\zeta(s_3)  \notag\\
=&\sum_{s_1,s_3\ge 2,s_2\ge 1} \bfx^{\bfs-\bfone}\sum_{
\substack{t_1\ge 2,t_2\ge 1
\\ t_1+t_2=s_1+s_3}}   \hspace{-.4cm}\binom{t_1-1}{s_3-1}
\zeta(t_1,t_2,s_2)\label{equ:triplerhs1}\\
+&\sum_{s_1,s_3\ge 2,s_2\ge 1} \bfx^{\bfs-\bfone}\sum_{
\substack{t_1\ge 2,t_2,t_3\ge 1
\\ t_1+t_2+t_3 \\
=s_1+s_2+s_3 }}
\hspace{-.2cm}
\binom{t_1-1}{s_1-1}\binom{t_2-1}{s_2-t_3}
\zeta(t_1,t_2,t_3)\label{equ:triplerhs2}\\
+&\sum_{s_1,s_3\ge 2,s_2\ge 1} \bfx^{\bfs-\bfone}\sum_{
\substack{t_1\ge 2,t_2,t_3\ge 1
\\ t_1+t_2+t_3 \\
=s_1+s_2+s_3 }}
\hspace{-.2cm}
\binom{t_1-1}{s_1-1}
\binom{t_2-1}{s_2-1}
\zeta(t_1,t_2,t_3).\label{equ:triplerhs3}
\end{align}
}
Similar to the last section we can get
{\allowdisplaybreaks
\begin{align*}
\eqref{equ:triplerhs1}
=&G_3(x_{13},x_1,x_2)-G_3(x_1,x_1,x_2)-\sum_{s_1\ge 2,s_2\ge 1}  x_1^{s_1-1} x_2^{s_2-1}
\zeta(s_1,1,s_2),\\
\eqref{equ:triplerhs2}
=&G_3(x_{13},x_{23},x_2)-G_3(x_3,x_{23},x_2)-G_3(x_1,x_2,x_2),\\
\eqref{equ:triplerhs3}
=&G_3(x_{13},x_{23},x_3)-G_3(x_3,x_{23},x_3)
    -\sum_{s_1\ge 2,s_2\ge 1} x_1^{s_1-1}x_2^{s_2-1} \zeta(s_1,s_2,1)
\end{align*}
}
On the other hand
{\allowdisplaybreaks
\begin{align}
 \sum_{s_1\ge 2,s_2\ge 1,s_3\ge 2}  \bfx^{\bfs-\bfone}\zeta(s_1,s_3,s_2)
=&G_3(x_1,x_3,x_2)-\sum_{s_1\ge 2,s_2\ge 1}  x_1^{s_1-1}x_2^{s_2-1}\zeta(s_1,1,s_2) \label{equ:stuffle1}\\
\sum_{s_1\ge 2,s_2\ge 1,s_3\ge 2}  \bfx^{\bfs-\bfone}\zeta(s_3,s_1,s_2)
=&G_3(x_3,x_1,x_2)-\sum_{s_1\ge 2,s_2\ge 1}  x_1^{s_1-1}x_2^{s_2-1}\zeta(s_1,1,s_2) \label{equ:stuffle2} \\
\sum_{s_1\ge 2,s_2\ge 1,s_3\ge 2}  \bfx^{\bfs-\bfone}\zeta(s_1,s_2,s_3)
=&G_3(x_1,x_2,x_3)-\sum_{s_1\ge 2,s_2\ge 1}  x_1^{s_1-1}x_2^{s_2-1}\zeta(s_1,s_2,1). \label{equ:stuffle3}
\end{align}
}
Also
{\allowdisplaybreaks
\begin{align}
\ &\sum_{s_1\ge 2,s_2\ge 1,s_3\ge 2}  \bfx^{\bfs-\bfone}\zeta(s_1+s_3,s_2)\notag\\
=&\sum_{s_2\ge 1}x_2^{s_2-1}\sum_{s\ge 2}
\left( \sum_{s_1\ge 2,s_3\ge 2,s_1+s_3=s} x_1^{s_1-1}x_3^{s_3-1}\right)  \zeta(s,s_2)\notag \\
=&\sum_{s_2\ge 1}x_2^{s_2-1} \sum_{s\ge 2}
\left( \frac{x_3^{s-1}-x_1^{s-1}}{x_3-x_1}-x_3^{s-2}-x_1^{s-2}+\gd_{s,2}\right)  \zeta(s,s_2) \notag\\
=&\frac{1}{x_3-x_1}\Big(G_2(x_3,x_2)-G_2(x_1,x_2)\Big)-\frac{1}{x_1}G_2(x_1,x_2)-\frac{1}{x_3}G_2(x_3,x_2)
+\sum_{s\ge 1}x_2^{s-1} \zeta(2,s) \label{equ:stuffle4}
\end{align}
}
where $\gd_{s,2}=1$ is $s=2$ and $\gd_{s,2}=0$ otherwise. Similarly
{\allowdisplaybreaks
\begin{align}
\ &\sum_{s_1\ge 2,s_2\ge 1,s_3\ge 2}  \bfx^{\bfs-\bfone}\zeta(s_1,s_2+s_3)\notag\\
=&\sum_{s_1\ge 2}x_1^{s_1-1}\sum_{s\ge 1}
\left( \sum_{s_2\ge 1,s_3\ge 2,s_2+s_3=s} x_2^{s_2-1}x_3^{s_3-1}\right)  \zeta(s_1,s) \notag\\
=&\sum_{s_1\ge 2}x_1^{s_1-1} \sum_{s\ge 1}
\left( \frac{x_3^{s-1}-x_2^{s-1}}{x_3-x_1}-x_2^{s-2}+\gd_{s,1}x_2^{-1}\right)  \zeta(s_1,s) \notag\\
=&\frac{1}{x_3-x_2}\Big(G_2(x_1,x_3)-G_2(x_1,x_2)\Big)-\frac{1}{x_2}G_2(x_1,x_2)
+\frac{1}{x_2} \sum_{s\ge 2}x_1^{s-1}  \zeta(s,1) \label{equ:stuffle5}
\end{align}
}
Notice we have the stuffle relation
\begin{equation*}
\zeta(s_1,s_2)\,\zeta(s_3)=\zeta(s_1,s_2,s_3)+\zeta(s_1,s_3,s_2)+\zeta(s_3,s_1,s_2)
+\zeta(s_1+s_3,s_2)+\zeta(s_1,s_2+s_3).
\end{equation*}
Hence the sum of \eqref{equ:stuffle1} through \eqref{equ:stuffle5} equals the sum of \eqref{equ:triplerhs1} throught \eqref{equ:triplerhs3}, giving the following result.

\begin{thm}\label{thm:dblZetaRiemannZeta}
We have
{\allowdisplaybreaks
\begin{multline}\label{equ:tripleID}
G_3(x_{13},x_1,x_2)-G_3(x_1,x_1,x_2)+G_3(x_{13},x_{23},x_2)
    -G_3(x_3,x_{23},x_2)-G_3(x_1,x_2,x_2)\\
+G_3(x_{13},x_{23},x_3)-G_3(x_3,x_{23},x_3)
-G_3(x_1,x_2,x_3)-G_3(x_1,x_3,x_2)-G_3(x_3,x_1,x_2)\\
=\frac{1}{x_3-x_1}\Big(G_2(x_3,x_2)-G_2(x_1,x_2)\Big)
+\frac{1}{x_3-x_2}\Big(G_2(x_1,x_3)-G_2(x_1,x_2)\Big)\\
-\frac{1}{x_1}G_2(x_1,x_2)-\frac{1}{x_3}G_2(x_3,x_2)-\frac{1}{x_2}G_2(x_1,x_2)\\
-\sum_{s_1\ge 2,s_2\ge 1}  x_1^{s_1-1}x_2^{s_2-1}\zeta(s_1,1,s_2)
+\sum_{s\ge 1}x_2^{s-1} \zeta(2,s)+\frac{1}{x_2} \sum_{s\ge 2}x_1^{s-1}  \zeta(s,1) .
\end{multline}
}
\end{thm}
\begin{proof}
Clear.
\end{proof}

\begin{thm}\label{thm:triple=dbl}
Let $a,b$ and $c$ be three real numbers. Then we have
{\allowdisplaybreaks
\begin{multline}\label{equ:triple=dbl}
\sum_{\substack{j\ge 2,k,l\ge 1\\ j+k+l=n}}
\Big[(a+c)^{j-1} a^k b^lc+a(a+c)^{j-1} (b+c)^{k-1}(cb^l+bc^l)\\
-a^{j+k-1} b^lc-a^j b^{k+l-1}c-ac^j (b+c)^{k-1} (b^l+bc^{l-1})
-a^j b^k c^l-a^j c^k b^l-c^j a^k b^l\Big] \zeta(j,k,l)\\
=\sum_{\substack{j\ge 2,k\ge 1\\ j+k=n}}
\Big[\frac{(ac^j-a^j c)b^k}{c-a}+\frac{a^j(bc^k-b^k c)}{c-b}
-a^{j-1} b^k c-a b^k c^{j-1}-a^j b^{k-1} c\Big] \zeta(j,k)\\
-\sum_{k=2}^{n-2} a^k b^{n-1-k}c \zeta(k,1,n-1-k)+ab^{n-2}c\zeta(2,n-2)+a^{n-1}c\zeta(n-1,1).
\end{multline}
}
\end{thm}

\begin{proof}
Multiplying $x_1x_2x_3$ on \eqref{equ:tripleID}, taking
$x_1=at,x_2=bt, x_3=ct$  and then comparing the coefficients of $t^n$
we arrive at \eqref{equ:triple=dbl} immediately.
\end{proof}

Notice that Theorem \ref{thm:dblZetaRiemannZeta} is very similar to \cite[Thm. 1.1(i)]{Machide2012} but not the same. Moreover, by comparing the two results we obtain the following immediately.
\begin{thm}\label{thm:triple=dblReduceGF}
We have
\begin{multline}\label{equ:triple=dblReduceGF}
G_3(x_1,x_1,x_2)+G_3(x_1,x_2,x_2)
-\frac{1}{x_1}G_2(x_1,x_2)-\frac{1}{x_2}G_2(x_1,x_2)\\
=\sum_{s_1\ge 2,s_2\ge 1}  x_1^{s_1-1}x_2^{s_2-1}\zeta(s_1,1,s_2)
-\sum_{s\ge 1}x_2^{s-1} \zeta(2,s)-\frac{1}{x_2} \sum_{s\ge 2}x_1^{s-1}  \zeta(s,1) .
\end{multline}
\end{thm}
\begin{proof}
The terms appearing in \eqref{equ:triple=dblReduceGF} are exactly those that are in
Theorem \ref{thm:dblZetaRiemannZeta} but not in \cite[Thm. 1.1(i)]{Machide2012}.
\end{proof}

\begin{thm}\label{thm:triple=dblReduce}
Let $a$ and $b$ be three real numbers. Then we have
\begin{multline}\label{equ:triple=dblReduce}
\sum_{\substack{j\ge 2,k,l\ge 1\\ j+k+l=n}}
\Big[a^{j+k-1} b^l+a^j b^{k+l-1}\Big] \zeta(j,k,l)
-\sum_{\substack{j\ge 2,k\ge 1\\ j+k=n}}
\Big[a^{j-1} b^k +a^j b^{k-1} \Big] \zeta(j,k)\\
=\sum_{k=2}^{n-2} a^k b^{n-1-k} \zeta(k,1,n-1-k)-ab^{n-2}\zeta(2,n-2)-a^{n-1}\zeta(n-1,1).
\end{multline}
\end{thm}
\begin{proof}
This follows from Theorem \ref{thm:triple=dblReduceGF} immediately.
\end{proof}
\begin{cor}\label{cor:triple=dbl}
Let $n$ be a positive integer such that $n\ge 3$. Then
\begin{equation}\label{equ:triple=dblReduce1}
\sum_{\substack{j\ge 2,k,l\ge 1\\ j+k+l=n}}
(2j+k)\zeta(j,k,l)-
\sum_{\substack{k\ge 2,l\ge 1\\ k+l=n-1}} k\zeta(k,1,l)=\zeta(2,n-2)+4\zeta(n)-(3n-5)\zeta(n-1,1).
\end{equation}
\end{cor}
\begin{proof}
Differentiating \eqref{equ:triple=dblReduce} with respect to $a$ and then putting $a=b=1$ we get
\begin{multline*}
\sum_{\substack{j\ge 2,k,l\ge 1\\ j+k+l=n}} (2j+k-1)\zeta(j,k,l)
-\sum_{\substack{j\ge 2,k\ge 1\\ j+k=n}} (2j-1) \zeta(j,k)\\
=\sum_{\substack{k\ge 2,l\ge 1\\ k+l=n-1}} k\zeta(k,1,l)-\zeta(2,n-2)-(n-1)\zeta(n-1,1).
\end{multline*}
The corollary now follows from \eqref{equ:G2Der} and the sum formulas \eqref{equ:sumConjMZV} for $d=2,3.$
\end{proof}

The following corollary provides a sum formula relating some special type triple zeta
and double zeta values. It also appeared in \cite{Machide2012} as (5.12)
which is a special case of \cite[Thm.~2.3]{HoffmanOh2003}.
\begin{cor}\label{cor:triple=dbl2}
Let $n$ be a positive integer such that $n\ge 3$. Then
\begin{equation}\label{equ:mid1}
\sum_{k=2}^{n-1} \zeta(k,1,n-k)=\zeta(2,n-1)+\zeta(n,1).
\end{equation}
\end{cor}
\begin{proof}
Taking $a=b=1$ and let $c\to 1$ in the Theorem \ref{thm:triple=dbl} we get
\begin{multline*}
\sum_{\substack{j\ge 2,k,l\ge 1\\ j+k+l=n}}
(2^{j-1}+2^{j+k-1}+2^k-5) \zeta(j,k,l)-
(n-5) \sum_{\substack{j\ge 2,k\ge 1\\ j+k=n}}\zeta(j,k)\\
= -\sum_{k=2}^{n-2} \zeta(k,1,n-1-k)+\zeta(2,n-2)+\zeta(n-1,1).
\end{multline*}
By sum formula \eqref{equ:sumConjMZV} we see that
\begin{equation*}
\sum_{\substack{j\ge 2,k\ge 1\\ j+k=n}} \zeta(j,k)
=\sum_{\substack{j\ge 2,k,l\ge 1\\ j+k+l=n}} \zeta(j,k,l)=\zeta(n).
\end{equation*}
Now \eqref{equ:mid1} quickly follows from the weighted sum formula of Guo and Xie
\cite[Theorem 1.1]{GuoXi2009} (setting $k=2$ there).
\end{proof}

\section{Depth 4: product of two double zetas}
In this last section we turn our attention to depth 4 case and derive some new families of
MZV identities using the idea of generating functions developed as above. Throughout this section
we set $\bfx^{\bfs-\bfone}=x_1^{s_1-1}x_2^{s_2-1}x_3^{s_3-1}x_4^{s_4-1}$. We also use the short hand
$t_{ij}=t_i+t_j$, $s_{ij}=s_i+s_j$ and so on.

First, for integers $s_1,s_3\ge 2$ and $s_2,s_4\ge 1$, Guo and Xie proved the following
at the end of \cite{GuoXi2009}
{\allowdisplaybreaks
\begin{align*}
& \zeta(s_1,s_2)\,\zeta(s_3,s_4) \notag\\
=& \sum\limits_{ \substack{t_1\ge 2,t_2,t_3\ge 1 \\ t_{123}=s_{123}}}
  \binom{t_1-1}{s_1-1}\binom{t_2-1}{s_2-1}\zeta(t_1,t_2,t_3,s_4)\\
+& \sum_{ \substack{t_1\ge 2,t_2,t_3\ge 1\\ t_{123}=s_{134}}}
 \binom{t_1-1}{s_3-1}\binom{t_2-1}{s_4-1} \zeta(t_1,t_2,t_3,s_2)\\
+& \sum_{ \substack{t_1\ge 2,t_2,t_3,t_4\ge 1\\ t_{1234}=s_{1234}}}
 \bigg [\binom{t_1-1}{s_1-1} \binom{t_2-1}{t_{12}-s_{13}}
 \bigg (\binom{t_3-1}{s_4-t_4} +  \binom{t_3-1}{s_4-1}\bigg) \notag \\
& \quad \quad \quad\quad \quad  + \binom{t_1-1}{s_3-1}\binom{t_2-1}{t_{12}-s_{13}}
 \bigg (\binom{t_3-1}{s_2-t_4}+\binom{t_3-1}{s_2-1}\bigg) \bigg ]  \zeta(t_1,t_2,t_3,t_4).
\end{align*}
}
Hence
 {\allowdisplaybreaks
\begin{align}
\lefteqn{\sum_{\substack{s_1, s_3\geq 2\\ s_2, s_4\geq 1}}
    \bfx^{\bfs-\bfone}\zeta(s_1,s_2)\,\zeta(s_3,s_4)} \notag\\
=& \sum_{\substack{s_1, s_3\geq 2\\ s_2, s_4\geq 1}}\bfx^{\bfs-\bfone}
\sum_{ \substack{t_1\ge 2,t_2,t_3\ge 1\\ t_{123}=s_{123}}}
     \binom{t_1-1}{s_1-1}\binom{t_2-1}{s_2-1}\zeta(t_1,t_2,t_3,s_4) \label{equ:2dblZ1}\\
+& \sum_{\substack{s_1, s_3\geq 2\\ s_2, s_4\geq 1}}
    \bfx^{\bfs-\bfone}\sum_{ \substack{t_1\ge 2,t_2,t_3\ge 1\\ t_{123}=s_{134} }}
    \binom{t_1-1}{s_3-1}\binom{t_2-1}{s_4-1}\zeta(t_1,t_2,t_3,s_2) \label{equ:2dblZ2}\\
+&\sum_{\substack{s_1, s_3\geq 2\\ s_2, s_4\geq 1}}\bfx^{\bfs-\bfone}
\sum_{ \substack{t_1\ge 2,t_2,t_3,t_4\ge 1\\ t_{1234}=s_{1234} }}
\binom{t_1-1}{s_1-1} \binom{t_2-1}{t_{12}-s_{13}}
 \binom{t_3-1}{s_4-t_4}  \zeta(t_1,t_2,t_3,t_4)
 \label{equ:2dblZ3}
\\
+&\sum_{\substack{s_1, s_3\geq 2\\ s_2, s_4\geq 1}}\bfx^{\bfs-\bfone}
\sum_{ \substack{t_1\ge 2,t_2,t_3,t_4\ge 1\\ t_{1234}=s_{1234} }}
\binom{t_1-1}{s_1-1} \binom{t_2-1}{t_{12}-s_{13}}
  \binom{t_3-1}{s_4-1}  \zeta(t_1,t_2,t_3,t_4)
 \label{equ:2dblZ4}
\\
+&\sum_{\substack{s_1, s_3\geq 2\\ s_2, s_4\geq 1}}\bfx^{\bfs-\bfone}
\sum_{ \substack{t_1\ge 2,t_2,t_3,t_4\ge 1\\ t_{1234}= \\ s_{1234} }}
\binom{t_1-1}{ s_3-1}\binom{t_2-1}{
t_{12}-s_{13}} \binom{t_3-1}{s_2-t_4}  \zeta(t_1,t_2,t_3,t_4)\label{equ:2dblZ5}\\
+&\sum_{\substack{s_1, s_3\geq 2\\ s_2, s_4\geq 1}}\bfx^{\bfs-\bfone}
\sum_{ \substack{t_1\ge 2,t_2,t_3,t_4\ge 1\\ t_{1234}=s_{1234} }}
\binom{t_1-1}{ s_3-1}\binom{t_2-1}{
t_{12}-s_{13}} \binom{t_3-1}{s_2-1}  \zeta(t_1,t_2,t_3,t_4).\label{equ:2dblZ6}
\end{align}
}
As before we can get
 {\allowdisplaybreaks
\begin{align*}
\eqref{equ:2dblZ1}
& =  G_4(x_{13}, x_{23}, x_3, x_4)-G_4(x_3, x_{23}, x_3, x_4)
 - \sum_{\substack{s_1\geq 2\\ s_2, s_4\geq 1}}x_1^{s_1-1}x_2^{s_2-1}x_4^{s_4-1}
\zeta(s_1, s_2, 1, s_4).\\
\eqref{equ:2dblZ3}
=& G(x_{13}, x_{23}, x_{24}, x_4)-G(x_3, x_{23}, x_{24}, x_4)- G(x_1, x_2, x_{24}, x_4) \\
\eqref{equ:2dblZ4}
=& G(x_{13}, x_{23}, x_{24}, x_2)- G(x_3, x_{23}, x_{24}, x_2)- G(x_1, x_2, x_{24}, x_2) \\
\eqref{equ:2dblZ2}
& = \gs_{x_1,x_3}\gs_{x_2,x_4} \eqref{equ:2dblZ1}, \quad
\eqref{equ:2dblZ5}
=\gs_{x_1,x_3}\gs_{x_2,x_4} \eqref{equ:2dblZ3}, \quad
\eqref{equ:2dblZ6}
= \gs_{x_1,x_3}\gs_{x_2,x_4} \eqref{equ:2dblZ4}.
\end{align*}}
On the other hand, by the stuffle relation
{\allowdisplaybreaks
\begin{align*}
&\zeta(s_1, s_2)\zeta(s_3, s_4) = \zeta(s_3, s_4, s_1, s_2)+ \zeta(s_3, s_1, s_4, s_2)+ \zeta(s_1, s_3, s_4, s_2)
 +\zeta(s_1, s_3, s_2, s_4)\\
& +\zeta(s_1, s_2, s_3, s_4)+ \zeta(s_3, s_1, s_2,s_4)+ \zeta(s_{13}, s_2, s_4)+ \zeta(s_1, s_{23}, s_4)+ \zeta(s_{13}, s_2,s_4)\\
&+ \zeta(s_3, s_{14}, s_2)+ \zeta(s_1, s_3, s_{24})+ \zeta(s_3, s_1, s_{24})+ \zeta(s_{13}, s_{24}).\\
\end{align*}
}
Thus we have
{\allowdisplaybreaks
\begin{align}
\ &\sum_{\substack{s_1, s_3\geq 2\\ s_2, s_4\geq 1}}
    \bfx^{\bfs-\bfone}\zeta(s_1,s_2)\,\zeta(s_3,s_4) \notag\\
=& \sum_{\substack{s_1, s_3\geq 2\\ s_2, s_4\geq 1}}\bfx^{\bfs-\bfone}
\zeta(s_3,s_4,s_1,s_2) + \sum_{\substack{s_1, s_3\geq 2\\ s_2, s_4\geq 1}}\bfx^{\bfs-\bfone}
\zeta(s_3,s_1,s_4,s_2)\label{equ:2sdblZ1}\\
+& \sum_{\substack{s_1, s_3\geq 2\\ s_2, s_4\geq 1}}\bfx^{\bfs-\bfone}
\zeta(s_1,s_3,s_4,s_2)+ \sum_{\substack{s_1, s_3\geq 2\\ s_2, s_4\geq 1}}\bfx^{\bfs-\bfone}
\zeta(s_1,s_3,s_2,s_4)\label{equ:2sdblZ2}\\
+&\sum_{\substack{s_1, s_3\geq 2\\ s_2, s_4\geq 1}}\bfx^{\bfs-\bfone}
\zeta(s_1,s_2,s_3,s_4)+ \sum_{\substack{s_1, s_3\geq 2\\ s_2, s_4\geq 1}}\bfx^{\bfs-\bfone}
\zeta(s_3,s_1,s_2,s_4)\label{equ:2sdblZ3}\\
+&\sum_{\substack{s_1, s_3\geq 2\\ s_2, s_4\geq 1}}\bfx^{\bfs-\bfone}
\zeta(s_{13}, s_4, s_2)+ \sum_{\substack{s_1, s_3\geq 2\\ s_2, s_4\geq 1}}\bfx^{\bfs-\bfone}
\zeta(s_1,s_{23},s_4)
 \label{equ:2sdblZ4}
\\
+&\sum_{\substack{s_1, s_3\geq 2\\ s_2, s_4\geq 1}}\bfx^{\bfs-\bfone}
\zeta(s_{13},s_2,s_4)+ \sum_{\substack{s_1, s_3\geq 2\\ s_2, s_4\geq 1}}\bfx^{\bfs-\bfone}
\zeta(s_3,s_4+s_1,s_2)\label{equ:2sdblZ5}\\
+&\sum_{\substack{s_1, s_3\geq 2\\ s_2, s_4\geq 1}}\bfx^{\bfs-\bfone}
\zeta(s_1,s_3,s_{24})+ \sum_{\substack{s_1, s_3\geq 2\\ s_2, s_4\geq 1}}\bfx^{\bfs-\bfone}
\zeta(s_3,s_1,s_{24})\label{equ:2sdblZ6}\\
+& \sum_{\substack{s_1, s_3\geq 2\\ s_2, s_4\geq 1}}\bfx^{\bfs-\bfone}
\zeta(s_{13},s_{24})\label{equ:2sdblz7}.
\end{align}
}
One can verify the following identities as before:
{\allowdisplaybreaks
\begin{align*}
\eqref{equ:2sdblZ1}
=& G_4(x_3, x_4, x_1, x_2)+G_4(x_3, x_1, x_4, x_2)-\sum_{\substack{s_3\geq 2\\s_2, s_4\geq 1}}x_2^{s_2-1}x_3^{s_3-1}x_4^{s_4-1}\zeta(s_3, s_4, 1, s_2)\\
& \hskip2cm -\sum_{\substack{s_3\geq 2\\s_2, s_4\geq 1}}x_2^{s_2-1}x_3^{s_3-1}x_4^{s_4-1}\zeta(s_3, 1, s_4, s_2)\\
\eqref{equ:2sdblZ2}
=& G_4(x_1, x_3, x_4, x_2)+G_4(x_1, x_3, x_2, x_4)-\sum_{\substack{s_1\geq 2\\s_2, s_4\geq 1}}x_1^{s_1-1}x_2^{s_2-1}x_4^{s_4-1}\zeta(s_1, 1, s_4, s_2)\\
& \hskip2cm -\sum_{\substack{s_1\geq 2\\s_2, s_4\geq 1}}x_1^{s_1-1}x_2^{s_2-1}x_4^{s_4-1}\zeta(s_1, 1, s_2, s_4)\\
\eqref{equ:2sdblZ3}
=& G_4(x_1, x_2, x_3, x_4)+G_4(x_3, x_1, x_2, x_4)-\sum_{\substack{s_1\geq 2\\s_2, s_4\geq 1}}x_1^{s_1-1}x_2^{s_2-1}x_4^{s_4-1}\zeta(s_1, s_2, 1, s_4)\\
& \hskip2cm -\sum_{\substack{s_3\geq 2\\s_2, s_4\geq 1}}x_2^{s_2-1}x_3^{s_3-1}x_4^{s_4-1}\zeta(s_3, 1, s_2, s_4)\\
\eqref{equ:2sdblZ4}
=& \bigoplus_{\calC(x_1,x_3)}\Big\{\frac{G_3(x_1, x_4, x_2)}{x_1-x_3}-\frac{G_3(x_1, x_4, x_2)}{x_1}\Big\}
+ \frac{G_3(x_1, x_2, x_4)- G_3(x_1, x_3, x_4)}{x_2-x_3} \\
& \hskip2cm -\frac{G_3(x_1, x_2, x_4)}{x_2}+ \sum_{s, t\geq 1}x_2^{s-1}x_4^{t-1}\zeta(2,t,s)
 +\frac{1}{x_2}\sum_{\substack{s\geq 2\\ t \geq 1}}x_1^{s-1}x_4^{t-1}\zeta(s, 1, t)\\
\eqref{equ:2sdblZ5}
=&  \gs_{x_1,x_3}\gs_{x_2,x_4} \eqref{equ:2sdblZ4}\\
\eqref{equ:2sdblZ6}
=& \bigoplus_{\calC(x_2,x_4)}\Big\{ \bigoplus_{\calC(x_1,x_3)}\Big\{\frac{G_3(x_1, x_3, x_2)}{x_2-x_4}- \frac{1}{x_2-x_4}\sum_{\substack{t\geq 2\\ s\geq 1}}x_1^{t-1}x_2^{s-1}\zeta(t, 1, s)\Big\}\Big\}\\
\eqref{equ:2sdblz7}
=& \bigoplus_{\calC(x_2,x_4)}\Big\{ \bigoplus_{\calC(x_1,x_3)}\Big\{
\left(\frac{1}{x_1-x_3}-\frac{1}{x_1}\right)\frac{G_2(x_1, x_2)}{x_2-x_4} \Big\}
+ \frac{1}{x_2-x_4}\sum_{t\geq 1} x_2^{t-1}\zeta(2, t)\Big\}.
\end{align*}
}
Hence the sum of \eqref{equ:2dblZ1} to \eqref{equ:2dblZ6} is equal to the sum of \eqref{equ:2sdblZ1} to \eqref{equ:2sdblz7}. This equality establishes an identity involving $G_4$ and four formal variables
$x_1,\dots,x_4$. By taking $x_1=at,x_2=bt,x_3=ct$ and $x_4=dt$ and comparing the coefficient of $t^{n-4}$ we can obtain our last theorem.

Let $S_1=\{e, \sigma_{a, c}\sigma_{b, d}\}$, where $\sigma_{a, c}$ (or $\sigma_{b, d}$) denotes the transposition that switches $a$ and $c$ (or $b$ and $d$) and $S_2=\{ e,\sigma_{a,c}\sigma_{b, d}, \sigma_{a,c}, \sigma_{b,c}, \sigma_{b,d}, \sigma_{c, d, b, a}\}.$ For any subset $S$ of the symmetric group $\mathfrak{S}_4,$ let $\displaystyle\bigoplus_{S} f(a, b, c, d)= \displaystyle\sum_{\sigma\in S} f(\sigma(a), \sigma(b), \sigma(c), \sigma(d)).$
\begin{thm}\label{thm:depth4}
Let $a, b$, $c$ and $d$ be any real numbers. Then for any positive integer $n\geq 2$, we have
{\allowdisplaybreaks
\begin{multline*}
\bigoplus_{S_1}\Big\{
\sum_{\substack{i\geq 2\\ i+j+k+l=n}}[abc(a+c)^{i-1}(b+c)^{j-1} -abc^{i}(b+c)^{j-1}
-ca^{i}b^{j}](b+d)^{k-1}d^{l}\cdot\zeta(i, j, k, l) \\
+ \sum_{\substack{i\geq 2\\ i+j+k+l=n}}[(a+c)^{i-1}-c^{i-1}]ab(b+c)^{j-1}c^{k}d^{l} \zeta(i,j, k, l)+\sum_{\substack{i\geq 2\\ i+j+l=n-1}}(b^{j}d^{l}+d^{j}b^{l}) ca^{i}\zeta(i, 1, j, l) \\
+\sum_{\substack{i\geq 2\\ i+j+k+l=n}}[acd(a+c)^{i-1}(b+c)^{j-1} -ad c^{i}(b+c)^{j-1}
-cda^{i}b^{j-1}](b+d)^{k-1}b^{l}\cdot\zeta(i, j, k, l)\Big\}\\
= \bigoplus_{S_2}\Big\{\sum_{\substack{i\geq 2\\ i+j+k+l=n}}a^{i}b^{j}c^{k}d^{l}\zeta(i, j, k, l) \Big\}
+\bigoplus_{S_1}\Big\{
\sum_{\substack{i\geq 2\\ i+j+k=n}}\left(\frac{cb^{j}-bc^{j}}{b-c} - cb^{j-1}\right)a^{i}d^{k}\zeta(i,j, k) \\
+\sum_{\substack{i\geq 2\\ i+j+k=n}}\left(\frac{ca^{i}-ac^{i}}{a-c}- ca^{i-1}
    - ac^{i-1} \right)b^jd^{k}\zeta(i, j, k)+ \sum_{\substack{i\geq 2\\ i+j+k=n}} \frac{db^{k}-bd^{k}}{b-d}a^{i}c^{j}\zeta(i,j,k)\\
+\sum_{j+k=n-2}acd^j b^k \zeta(2, j, k) + \sum_{\substack{i\geq 2\\i+k=n-1}}c a^{i}d^{k}\zeta(i, 1, k)
    - \sum_{\substack{i\geq 2\\ i+k=n-1}} \frac{db^{k}-bd^{k}}{b-d}a^{i}c  \zeta(i, 1, k)\Big\}\\
+ \sum_{\substack{i\geq 2,\\ i+j=n}}\left(\frac{ca^{i}-ac^{i}}{a-c}-ca^{i-1}-ac^{i-1}\right)
\frac{db^{j}-bd^{j}}{b-d}\zeta(i, j) + \frac{ac(db^{n-2}-bd^{n-2})}{b-d}\zeta(2, n-2).
\end{multline*}
}
\end{thm}

\begin{cor}\label{cor:depth4}
Let $n\geq 5$ be any positive integer. Then
\begin{multline}\label{equ:corFinal}
2\sum_{\substack{i\geq 2\\i+j+l=n-1}}\zeta(i, 1, j, l)- \sum_{\substack{i\geq 2\\i+j+k+l=n}}(2^{i-1}+2^k)\zeta(i, j, k, l)-\sum_{\substack{k\geq 2\\k+j=n-1}}k\zeta(k, 1, j)\\
= 2\zeta(2, n-2)+ (n-3)\zeta(n-2,2)-(2n-5)\zeta(n-1,1)+\frac{n+5}{2} \zeta(n).
\end{multline}
\end{cor}
\begin{proof}
Let $a=1=b$ and $c\to 1$ and $d\to 1$ in Theorem~\ref{thm:depth4}. Then we get
\begin{multline*}
2\sum_{\substack{i\geq 2\\ i+j+k+l=n}}(2^{i+j+k-2}- 2^{j+k-1}- 2^k+2^{i+j-2}-2^{j-1}) \zeta(i, j, k, l) = 6\sum_{\substack{i\geq 2\\i+j+k+l=n}}\zeta(i, j, k, l)\\ - 4\sum_{\substack{i\geq 2\\i+j+l=n-1}}\zeta(i, 1, j, l)+2\sum_{\substack{i\geq 2\\i+j+k=n}}(n-6)\zeta(i, j, k) + 2\sum_{j+k=n-2}\zeta(2, j, k)\\
-2\sum_{\substack{i\geq 2\\i+k=n-1}}(k-2)\zeta(i, 1, k) + \sum_{\substack{i\geq 2\\i+j=n}}(i-3)(j-1)\zeta(i, j) + (n-3)\zeta(2, n-2).
\end{multline*}
By the sum formula and the weighted sum formula of Guo and Xie \cite[Theorem 1.1]{GuoXi2009} (setting $k=4$ there), we get
\begin{multline}\label{equ:cormiddlestep}
2\sum_{\substack{i\geq 2\\i+j+l=n-1}}\zeta(i, 1, j, l)- \sum_{\substack{i\geq 2\\i+j+k+l=n}}(2^{i-1}+2^k)\zeta(i, j, k, l)= \sum_{j+k=n-2}\zeta(2, j, k)-3\zeta(n)\\
-\sum_{i=2}^{n-2}(n-i-3)\zeta(i, 1, n-i-1) + \frac{1}{2}\sum_{i=2}^{n-1}(i-3)(n-i-1)\zeta(i, n-i) + \frac{n-3}{2}\zeta(2, n-2)
\end{multline}
Taking $l=2, i_1=2, i_2=n-3$ in \cite[Thm.\ 5.1]{Hoffman1992} we get
\begin{equation}\label{equ:Hoffl2case}
\sum_{j+k=n-2}\zeta(2, j, k)=\zeta(3,n-3)+\zeta(2,n-2).
\end{equation}
Combining this with \eqref{equ:G2Der}, \eqref{equ:G2DerDer}, \eqref{equ:HoffSumFormula2} and the sum formula
\eqref{equ:EulerSumFormula} we see easily that \eqref{equ:cormiddlestep} can be simplified to \eqref{equ:corFinal}.
This finishes the proof of the corollary.
\end{proof}

\bigskip
Haiping Yuan

Department of Mathematics, York College of Pennsylvania, York, PA 17403

hyuan@ycp.edu

\medskip
Jianqiang Zhao

Kavli Institute for Theoretical Physics China, Beijing, China and

Department of Mathematics, Eckerd College, St. Petersburg, FL 33711

zhaoj@eckerd.edu

\end{document}